\documentclass[12pt,a4paper]{amsproc}
\usepackage[english]{babel}
\usepackage{latexsym}
\usepackage{amsmath}
\usepackage{amssymb}
\usepackage[all]{xy}
\usepackage{bbding}
\usepackage{adforn}
\usepackage{amscd}
\usepackage[cp850]{inputenc}
\usepackage[mathscr]{eucal}
\newcommand{\To}{\longrightarrow}

\theoremstyle{plain}
\newtheorem{quest}{Question}[section]

\newtheorem{defin}[quest]{Definition}
\newtheorem{theorem}[quest]{Theorem}

\newtheorem{prop}[quest]{Proposition}
\newtheorem{corollary}[quest]{Corollary}
\newtheorem{lemma}[quest]{Lemma}

\theoremstyle{remark}

\newcommand{\SSS}{\mathbb{S}}
\newcommand{\R}{\mathbb{R}}
\newcommand{\N}{\mathbb{N}}

\newcommand{\C}{\mathbb{C}}

\newcommand{\sn}{\sum_{k=1}^n}
\newcommand{\mhor}{\Omega^{\prime}}
\newcommand{\adef}{\begin{defin}}
\newcommand{\zdef}{\end{defin}}

\def\Ext{\operatorname{Ext}}

\def\supp{\operatorname{supp}}

\def\PO{\operatorname{PO}}
\newcommand{\Dom}{\mathrm{Dom}}

\def\Ext{\operatorname{Ext}}

\newcommand{\aproof}{\begin{proof}}
\newcommand{\zproof}{\end{proof}}

\newcommand{\lop}{\curvearrowright}
\newcommand{\T}{\mathbb{T}}
\newcommand{\KP}{{\sf K}\hspace{-1pt}{\sf P}}
\newcommand{\F}{\mathbb{F}}

\title[Actions on twisted sums]{Group actions on twisted sums of Banach spaces}

\author{Jes\'{u}s M. F. Castillo}
\address{Instituto de Matem\'aticas Imuex\\ Universidad de Extremadura\\
Avenida de Elvas\\ 06071-Badajoz\\ Spain} \email{castillo@unex.es}

\author{Valentin Ferenczi}

\address{Departamento de Matem\'atica, Instituto de Matem\'atica e
Estat\'\i stica, Universidade de S\~ao Paulo, rua do Mat\~ao 1010,
05508-090 S\~ao Paulo SP, Brazil  \\ and \newline
Equipe d'Analyse Fonctionnelle \\
Institut de Math\'ematiques de Jussieu \\
Sorbonne Universit\'e - UPMC \\
Case 247, 4 place Jussieu \\
75252 Paris Cedex 05 \\
France.}
\email{ferenczi@ime.usp.br}



\thanks{The research of the first author was supported in part by  MINCIN Project PID2019-103961GB-C21, Spain, and
Project IB20038 de la Junta de Extremadura. The research of the second author was supported by FAPESP, grants 2013/11390-4,
2015/17216-1, 2016/25574-8 and by CNPq, grants 303034/2015-7 and 303731/2019-2}

\begin{document}

\begin{abstract} We study bounded actions of groups and semigroups $G$ on exact sequences of Banach spaces from the point of view of quasilinear maps, characterize the actions on the twisted sum space by commutator estimates and introduce the associated notions of $G$-centralizer and $G$-equivariant map. We will show that when (A) $G$ is an  amenable group and (U) the target space is complemented in its bidual by a $G$-equivariant projection, then uniformly bounded compatible families of operators generate bounded actions on the twisted sum space; that compatible
quasilinear maps are linear perturbations of $G$-centralizers; and that, under (A) and (U), $G$-centralizers are bounded perturbations of $G$-equivariant maps. The previous results are optimal. Several examples and counterexamples are presented involving the action of the isometry group of $L_p(0,1), p\neq 2$ on the Kalton-Peck space $Z_p$, certain non-unitarizable triangular representations of the free group $\F_\infty$ on the Hilbert space, the compatibility of complex structures on twisted sums, or bounded actions on the interpolation scale of $L_p$-spaces. In the last section we consider the category of $G$-Banach spaces and study its exact sequences, showing that, under (A) and (U), $G$-splitting and usual splitting coincide.\end{abstract}
\maketitle


\section{Introduction} This paper emerges from the observation of similarities between different problems:
\begin{itemize}
\item[(a)] The construction of non-unitarizable, bounded, representations of  the free group $\F_\infty$ on the Hilbert space. \item[(b)] The construction of operators on the Kalton-Peck space $Z_2$.
\item[(c)] The differential process associated to a complex interpolation scheme.
\item[(d)] Actions of groups on exact sequences of Banach spaces.
\item[(e)] The existence of certain bounded groups of isomorphisms on the space $c_0$.
\end{itemize}

In all cases, certain non-linear maps (including sometimes linear unbounded maps) and their compatibility with the action of some groups of operators through commutator estimates are at the core of the problem.
In (a), a linear unbounded map used to define a non-inner derivation and therefore a non-unitarizable representation \cite{PS}; in (b) the Kalton-Peck map $\KP$ \cite{KP}; in (c) is the ``$\Omega$-operator" mentioned by several authors Cwikel et al. \cite{CJMR}, Rochberg \cite{rochberg}, Carro \cite{Carro}... And in (d) we encounter the Banach version of the three-representation problem (see \cite{kuch}).
Another unexpected example (e) is a linear unbounded map used in \cite{AFGR} to define a non-trivial derivation in a study of bounded groups acting on $c_0$.
Connections between some of those elements had been observed before: for instance, Kalton observed \cite{kaltonams,K} that while working on K\"othe spaces, $\Omega$-operators are a special type of quasilinear map, that he called $L_\infty$-centralizers, intimately connected with the complex interpolation scale.

To obtain a unified point of view we consider a group or semigroup $G$, two bounded actions $u, v$ on two Banach spaces $X,Y$ and introduce the notion of $G$-centralizer $\Omega$: this allows us to construct an exact sequence $0\To X \to X\oplus_\Omega Y \To Y \To 0$ of Banach spaces and connect possible actions of $G$ on the twisted sum space $X\oplus_\Omega Y$  with commutator estimates involving $\Omega$ and derivations of the group.

Our results move at two levels, the theoretical and the examples/counterexamples. On the former side, the theory we display could be described as follows. Let $(A)$ be the condition: $G$ is amenable and let $(U)$ be: $X$ is $G$-complemented in its bidual (meaning, complemented by a $G$-equivariant projection).
\begin{itemize}
\item Triangular representations of groups on the Hilbert space $H$ may be interpreted as diagonal representations on $H$ seen as a twisted Hilbert space.
\item Under $(A)$ and $(U)$, a uniformly bounded family $(T_g)_{g\in G}$ of operators yielding commutative diagrams
$$\begin{CD}
0@>>> X @>>> X\oplus_\Omega Y @>>> Y @>>> 0\\
&&@V{u(g)}VV @V{T_g}VV @VV{v(g)}V\\
0@>>> X @>>> X\oplus_\Omega Y @>>> Y @>>> 0\end{CD}$$
provides a compatible action of $G$ on $X\oplus_\Omega Y$.
\item Every $\Omega$ compatible with an action on $X\oplus_\Omega Y$ is a linear perturbation of a $G$-centralizer (possibly with values in a larger target space).
\item Under $(A)$ and $(U)$, every $G$-centralizer is a bounded perturbation of a $G$-equivariant map.
\item We introduce the category of $G$-Banach spaces and show that, under $(A)$ and $(U)$, a $G$-exact sequence of $G$-spaces $G$-splits if and only if it splits as an exact sequence of Banach spaces.
\end{itemize}

The results above are optimal because on the side of counterexamples:
\begin{itemize}
\item We will use a construction of Pytlic and Szwarc \cite{PS} to show a centralizer that is not a bounded perturbation of an equivariant centralizer when $G$ is non-amenable. We will provide another counterexample, inspired from \cite{AFGR}, when $X$ is not complemented in its bidual.
\item We will show that the Kalton-Peck map is not a centralizer for the groups of isometries on $L_p, p \neq 2$ or isometries preserving disjointness on $L_2$. It is however compatible with the actions of those groups.
\item In the case of the group of isometries of $L_2$, the Kalton-Peck map is not even compatible with the action of that group.
\end{itemize}

There are specific sections devoted to actions of groups on complex interpolation scales, on Kalton-Peck spaces and on higher order Rochberg spaces, as well as to the connections between $G$-centralizers and almost transitivity.

\section{The Background}

Let $X, Y$ be Banach spaces. In what follows $\Delta\subset Y$ represents a dense subspace of $Y$ (sometimes called the \emph{intersection} space), while $\Sigma$ represents the \emph{ambient} space, namely, a vector space containing $X$. When necessary, we will alternatively assume that there is an injective linear map $\jmath: X\to \Sigma$, in which case the subspace $\jmath[X]$ will be normed with $\|\jmath(x)\|=\|x\|_X$. A homogeneous map $\Omega: \Delta \To \Sigma$ is a $z$-linear map $\Delta \lop X$ if there is a constant $C$ such that for all finite sequences of elements $y_1, \dots, y_N \in \Delta$
\begin{itemize}
\item[(a)] $\Omega(\sum_{n=1}^N y_n)- \sum_{n=1}^N \Omega(y_n)\in \jmath[X]$
\item[(b)] $\|\Omega(\sum_{n=1}^N y_n)- \sum_{n=1}^N \Omega(y_n)\|_{\jmath[X]}\leq C\sum_{n=1}^N \|y_n\|_Y $ .
\end{itemize}
In this paper we mainly use the notation $\Omega: \Delta \lop X$, although $\Omega: Y \lop X$ can also be appear when the choice of $\Delta$ is clear from the context or irrelevant. When condition (b) holds only for pairs of points then $\Omega$ is called quasilinear. 
A quasilinear map $\Omega: \Delta \lop X$ with ambient space $\Sigma$ is said to be trivial if there is a linear (not necessarily continuous) map $L: \Delta \To \Sigma$ such  that $\Omega-L: \Delta \to \jmath[X]$ is bounded, in the sense that $\|\Omega(y)-L(y)\|_{\jmath[X]}\leq M\|y\|_Y$ for some constant $M$ and all $y\in \Delta$. Two quasilinear maps $\Phi, \Psi: \Delta \lop X $ with ambient space $\Sigma$ are said to be equivalent , and denoted $\Phi\sim \Psi$, (resp. boundedly equivalent) if $\Phi - \Psi$ is trivial (resp. $\Phi - \Psi:\Delta \To X$ is bounded). The twisted sum generated by a quasilinear map $\Omega$ is the completion $X \oplus_\Omega Y$ of the space $ X \oplus_\Omega \Delta:= \{(\omega, y) \in \Sigma \times \Delta: \omega - \Omega y \in \jmath[X]\}$ endowed with the quasi-norm $\|y\|_{Y}+\|\omega - \Omega y\|_{\jmath[X]}$. From now on, except when in need, we shall omit the embedding $\jmath$. If $\Omega$ is $z$-linear then $\|\cdot\|_\Omega$ is equivalent to a norm, and thus $X\oplus_\Omega Y$ is a Banach space. Kalton showed \cite{kalt} that quasilinear maps on $B$-convex Banach spaces are $z$-linear; therefore, twisted sums in which the quotient space is $B$-convex are Banach spaces. The map $\imath:X\To X\oplus_\Omega Y$ given by $\imath(y)=(y,0)$ is an into isometry and the map $\pi:X\oplus_\Omega Y\To Y$ given by $\pi(\omega,x)=x$ takes the unit ball of $X\oplus_\Omega Y$ onto that of $X$. These spaces and operators form a short exact sequence $\xymatrix{
0\ar[r] & X \ar[r]^-\imath & X\oplus_\Omega Y \ar[r]^-\pi & Y\ar[r] & 0}$ that shall be referred to as the sequence generated by $\Omega$. Two exact sequences of Banach spaces are called \emph{equivalent} when there is an operator $T$ making the diagram

$$\xymatrixrowsep{1pc}\xymatrix{&&Z\ar[dr]\ar[dd]^T&\\
0\ar[r]&X \ar[ur]\ar[dr]&&Y\ar[r]&0\\
&& Z'\ar[ur]&&}$$
commute. When $Z=X\oplus_\Omega Y$ and $Z'=X\oplus_\Phi Y$ that happens if and only if $\Phi$ and $\Psi$ are equivalent maps.

Given two maps $S, T$, its commutator is defined as $[S,T]= ST - TS$ provided this makes sense. We will need to use a generalized commutator for three maps defined as $[u,\Omega, v] = u\Omega  -\Omega v$, whenever this makes sense.

\section{G-centralizers}
\adef Let $G$ be a semigroup. A $G$-space is a normed space $X$ equipped with a bounded action $G\times X \to X$; namely, a morphism of semigroups  $u: G\to \mathfrak L(X)$ such that $\gamma(u) := \sup \{\|u(g)\|: g\in G\}<\infty$.\zdef

Note that we do not require $G$ to carry any topology and therefore there is no continuity involved with respect to $G$ (alternatively we may think of $G$ as discrete). Occasionally we will consider unbounded or even nonlinear actions, but that will be explicitly said. Paramount examples of bounded actions are (see the appropriate section in the paper for unexplained terms): (a) The action of the group of units $\mathcal U$  of $L_\infty(S, \mu)$ on either $L_\infty$-Banach modules or K\"othe spaces. In particular, the action of the Cantor group $2^\omega=\{-1,+1\}^\N$ on spaces with unconditional basis or that of the group $2^{<\omega}$ of elements of $2^\omega$ that are eventually $1$ on $c$. (b) The action of the group generated by measure preserving rearrangements of the base space and change of signs on rearrangement invariant K\"othe spaces. (c) The action of the group ${\rm Isom}(X)$ of isometries of $X$ on $X$. (d) The action of the group ${\rm Isom_{disj}}(L_2)$ of isometries that preserve disjointness on $L_2$. (e) The natural left regular action of the free group $\F_\infty$ on the Hilbert space seen as $\ell_2(\F_\infty)$.

Given an exact sequence $0\to X\to Z\to Y\to 0$ of $G$-spaces, we will agree for the rest of this paper that the action of $G$ on $X$ will be denoted $u$, that on $Y$ will be denoted $v$ and that on $Z$ will be denoted $\lambda$.

\adef\label{listofdef} Let $G$ be a semigroup. \begin{description}
\item[$G$-operator] An operator (resp. a linear map) $T: X\to Y$ between two $G$-spaces $X$ and $Y$ is a $G$-operator (resp. a $G$-linear map) if $v(g)T=Tu(g)$ for all $g \in G$.
\item[$G$-subspace] A $G$-subspace $Y'$ of $Y$ is a subspace of $Y$ such that the canonical inclusion $\imath: Y'\to Y$ is a $G$-operator; in which case we shall also occasionally say that $Y$ is a $G$-superspace of $Y'$.
\item[$G$-centralizer] Let $Y,X$ be $G$-spaces, let $\Delta \subset Y$ be a dense $G$-subspace of $Y$, and let $\Sigma \supset X$ be a $G$-superspace of $X$.  A quasilinear map $\Omega: \Delta \lop X$ with ambient space $\Sigma$ is said to be a $G$-centralizer if the family of maps $[u(g),\Omega,v(g)]$ takes values in $X$ and is uniformly bounded, i.e., there exists a constant $G(\Omega)>0$ such that $\|u(g)\Omega y-\Omega v(g) y\|_X \leq G(\Omega) \|y\|_Y$ for all $g \in G$ and $y\in \Delta$.
\end{description}\zdef

To avoid confusion, let us make explicit that in the above we use the same letter for an action on a $G$-space and for the action by restriction on a $G$-subspace; for example for any $g \in G$, $u(g)$ extends to a map on $\Sigma$ still denoted $u(g)$.

It will spare us a few headaches to briefly discuss the roles of the ``ambient" and ``intersection" spaces $\Sigma$ and $\Delta$. Observe that $\Omega$ is in principle only defined on $\Delta$, not in $Y$.  It is well known \cite{KP} that every quasilinear map $\Omega: \Delta \lop  X$ can be extended to a quasilinear map $\widehat{\Omega}: Y\to X$, but replacing $\Omega$ by this ``artificial" $\widehat{\Omega}$ may spoil the compatibility conditions with $G$, so this approach is useless for us.\\

\noindent \textbf{The issue of the ambient space.} Assume that one has two quasilinear maps $\Omega, \Phi: Y\lop X$, one taking values in the ambient space $\Sigma$ with embedding $\jmath: X \to \Sigma$ and the second in the ambient space $\Xi$ with embedding $\imath: X\to \Xi$. If $\PO$ denotes the pushout space (see e.g. \cite{hmbst}) $\PO= (\Sigma\oplus \Xi)/ X_0$ where $X_0=\{(\jmath x,-\imath x): x\in X\}$ then the maps $\sigma: \Sigma\To \PO$ and
$\xi: \Xi\To \PO$ given by $\sigma(z)=(z, 0)+ X_0$ and $\xi(z)=(0,z)+X_0$ yield a commutative diagram

$$\xymatrixrowsep{0,5pc}\xymatrix{
&\Sigma\ar[dr]^\sigma&\\
X \ar[ur]^\jmath\ar[dr]_\imath&&\PO\\
&\Xi\ar[ur]_\xi}$$
Thus, replacing $\Omega$ by $\sigma\Omega: Y\lop \sigma\jmath [X]$ and $\Phi$ by $\xi\Phi: Y\lop \xi\imath[X]$, and calling $\mathcal X= \sigma\jmath[X]=\xi\imath[X]$, then $\sigma\Omega$ and $\xi\Phi$ are quasilinear maps $Y\lop \mathcal X$ with ambient space $\PO$. We can therefore extend the equivalence notion to quasilinear maps with different ambient spaces, maintaining the notation: $\Omega\sim \Phi$ if and only if  $\sigma \Omega \sim \xi\Phi$. The modification is acceptable since $\Omega \sim \Phi$ if and only if they generate equivalent exact sequences: if $B= \sigma\Omega - \xi\Phi - L: Y\to \mathcal X$ is bounded for some linear map $L: Y\to \PO$ then the following sequences are equivalent
$$\xymatrixrowsep{1pc}\xymatrix{&&\mathcal X\oplus_{\sigma\Omega} Y \ar[dr]\ar[dd]^T&\\
0\ar[r]&\mathcal X\ar[ur]\ar[dr]&&Y\ar[r]&0\\
&& \mathcal X\oplus_{\xi\Phi} Y\ar[ur]&&}$$
via the operator $T(\sigma\jmath x, y)= (\sigma\jmath x - Ly, y)$: indeed, $(\sigma\jmath x - Ly, y)\in \mathcal X\oplus_{\xi\Phi} Y$ because $\sigma\jmath x - Ly - \xi\Phi y = \sigma\jmath x - \sigma\Omega y + By \in \mathcal X$ since  $B: Y\to \mathcal X$.
It is then a simple calculation that
$\|(\sigma\jmath x - Ly, y)\|_{\xi\Phi}$ and
 $\| (\sigma\jmath x, y)\|_{\sigma\Omega}$ are equal up to a multiplicative constant. Given $\Omega: Y\lop X$ with ambient space $\Sigma$ we can choose as ambient space $X\oplus_\Omega Y$ and replace $\Omega$ by $\mhor y =(\Omega y, y)$ to get

\begin{lemma}\label{perturlin} $\Omega \sim \mhor$. More precisely, there is a linear map $\mathcal L: \Delta \To \PO$ such that
$$\xi \mhor = \sigma \Omega + \mathcal L.$$\end{lemma}
\begin{proof} We just consider the commutative diagram

$$\xymatrixrowsep{1pc}\xymatrix{
&\Sigma\ar[dr]^\sigma&\\
X \ar[ur]^\jmath\ar[dr]_\imath&&\PO\\
&X\oplus_{\Omega}Y \ar[ur]_\xi}$$
 and keep track of what $\sigma,\xi,\imath$ do; namely, $\imath(x)=(x,0)$, $\sigma(\omega) = (0, \omega)$ and $\xi(\omega, y) = ((\omega, y), 0)$. Therefore $\sigma\Omega(y)=((0,0), \Omega y)$ and $\xi\mhor (y)= ((\Omega y, y), 0)$. A linear selection $Y\to X\oplus_\Omega Y$ for the natural quotient map has the form $y \to (\ell y, y)$ for some linear map $\ell: Y\to \Sigma$ such that $\Omega y - \ell y \in X$. If we define the linear map $\mathcal L: Y\to \PO$ given by $\mathcal Ly= [(\ell y, y), -\ell y)]$ then we have
\begin{eqnarray*}
\xi \mhor (y) - \sigma \Omega(y)  - \mathcal L(y) &=&  ((\Omega y, y), 0) - ( (0,0), \Omega y) - (\ell y, y), -\ell y)\\
 &=&  ((\Omega y - \ell y, 0), -( \Omega y - \ell y ) )\\
 &=&0
\end{eqnarray*}
since all elements $((x,0), -x)$ with $x\in X$ are $0$ in $\PO$.\end{proof}

When the spaces $\Sigma, \Xi$ carry their own norms and $X$ is a $G$-subspace of both $\Sigma$ and $\Xi$ then we must set $\PO= (\Sigma\oplus \Xi)/ \overline{\Delta}$ to make $\PO$ a $G$-superspace of $X$. However, once actions are involved, a situation appears: given an operator $u: X\to X$ and a quasilinear map $\Omega: \Delta \rightarrow \Sigma$ the composition $u\Omega$ seems impossible. A way to overcome the difficulty is to assume that $u:X\to X$ is the (continuous) restriction of a linear map $\Sigma \to \Sigma$. This is reasonable and, in most occasions, feasible; therefore we usually assume that $\Sigma$ is a $G$-superspace of $X$, as in the definition of $G$-centralizer. \medskip

\noindent\textbf{The issue of the dense subspace}. In classical interpolation theory one considers choices of $\Delta$ so that $\Omega: \Delta\to X$. Adapting their terminology, we can define the \emph{dominion} of quasilinear map $\Omega: Y\lop X$ as the space $\Dom \Omega = \{ y\in Y: \Omega y \in X\}$ endowed with the quasinorm $\|y\|_D = \|\Omega y \|_X + \|y\|_Y$. In this form $\Dom \Omega$ is isometric to the closed subspace $\{(0,y)\in X\oplus_\Omega Y\}$ of $X\oplus_\Omega Y$. More often than not, $\Dom\Omega$ is dense in $Y$, as it is the case in the complex interpolation context (this explains why we impose the assumption on the interpolation couple $(X_0, X_1)$ of being \emph{regular}, which means that $X_0\cap X_1$ is dense in both $X_0$ and $X_1$) and $\Dom \Omega = Y$ if and only if $\Omega: Y  \to X$ is bounded. On the other hand, it may well happen that $\Dom \Omega =\{0\}$: see \cite{noncom}, Proposition \ref{thekp} or the example after Proposition \ref{222}.

And again, when an action $v$ of $G$ on $Y$ is involved, we need a sound meaning for $\Omega v(g)$, which is achieved by guaranteeing that $v$ leaves $\Delta$ invariant. Still a problem appears when one has two quasilinear maps $\Omega: \Delta\lop X$ and $\Phi: \Delta'\lop X$ defined on  different dense subspaces $\Delta, \Delta'\subset Y$. In this case we cannot consider them defined on the same dense subspace by  making a simple intersection since it could well be that $\Delta\cap \Delta'=\{0\}$. In most cases the choice of a common $\Delta$ is natural, but, in general, one has to be careful with this point.\\

We are ready to start our study with a simple observation whose interest will be shown later: in the Banach ambient, the action of a group $G$ induces an action of the semigroup $\ell_1(G)$.

\begin{lemma} \label{1G} A $z$-linear $G$-centralizer between $G$-Banach spaces is an $\ell_1(G)$-centralizer between the associated $\ell_1(G)$-modules.
\end{lemma}
\begin{proof} Indeed, given $\sum_i |\lambda_i| = 1$,
$$\|\Omega\left (\sum_ i \lambda_i u(g_i) y\right ) -\sum_i \lambda_ i v(g_i) \Omega y\|
\leq \| \sum_i \lambda_i \Omega(u(g_i) y) -\sum_i \lambda_ i v(g_i) \Omega y\| + \sum_i \lambda_i \|u(g_i) y\|$$
$$\leq (G(\Omega)+1) \max \{\gamma(u), \gamma(v)\}\|y\|.\qedhere$$
\end{proof}

Our first objective is the three-representation problem that Kuchment considers in \cite{kuch}: given an exact sequence $0 \to X \to Z \to Y \to 0$ and some group $G$ acting on $Y, Z$ and $X$ in a compatible way, to what extent the action on $Z$ can be recovered from the actions on $X$ and $Y$. Or else: given $u, v$, how to obtain a compatible action $\lambda$ on $X\oplus_\Omega Y$?
\adef Let $0\to X\to Z\to Y\to 0$ be an exact sequence. Assume that $X,Y$ are $G$-spaces. An action $\lambda$ of $G$ on $Z$ will be called compatible with the sequence if for each $g\in G$ there is a commutative diagram
$$\begin{CD}
0@>>> X @>>> Z @>>> Y @>>> 0\\
&&@V{u(g)}VV @V{\lambda(g)}VV @VV{u(g)}V\\
0@>>> X @>>> Z @>>> Y @>>> 0\end{CD}$$\zdef

\

Compatibility is a homological notion: $G$ is compatible with a sequence if and only if its it compatible with any equivalent sequence. The existence of compatible actions and $G$-centralizers are connected:

\begin{prop}\label{baba} Let $0\to X\to Z\to Y\to 0$ be an exact sequence in which $X,Y$ are $G$-spaces. $\Omega: \Delta \lop X$ is a $G$-centralizer if and only if the diagonal action $g \mapsto \lambda(g) = \begin{pmatrix} u(g) & 0 \\ 0 & v(g) \end{pmatrix}$ on $X \oplus_\Omega Y$ is compatible and bounded.
\end{prop}
\begin{proof} Observe that the action $\lambda$ is defined first on $X \oplus_\Omega \Delta$ and then extended by density to $X \oplus_\Omega Y=\overline{X \oplus_\Omega \Delta}$.
Now, if $\Omega$ is a $G$-centralizer and $\|(x,y)\|_\Omega\leq 1$, then
\begin{eqnarray*}
\|\lambda(g)\|&\leq& \sup\|( u(g)x,v(g)y) \|_{\Omega} \\
&=& \sup \|u(g)x - \Omega v(g)y \|_X + \|v(g)y \|_Y \\
&=& \sup \|u(g)x - u(g)\Omega y + u(g)\Omega y - \Omega v(g)y \|_X + \|v(g)y \|_Y \\
&=& \sup \|u(g)\|\|x - \Omega y \_X| + \| u(g)\Omega y - \Omega v(g)y \|_X + \|v(g)\|\|y \|_Y\\
&\leq & \max \{\gamma(u), \gamma(v) \}  + \sup \| u(g)\Omega y - \Omega v(g)y \|_X\\
&\leq & \max \{\gamma(u), \gamma(v) \} + G(\Omega).\end{eqnarray*}
On the other hand, the best possible value of $G(\Omega)$ is at most $\gamma(\lambda)$ since $
\| u(g)\Omega y - \Omega v(g)y \|_X = \|\lambda(g)(\Omega y, y)\|_\Omega \leq \|\lambda(g)\| \|y\|_Y$. \end{proof}
.

\begin{lemma}\label{babal} Let $0\to X\to X\oplus_\Omega Y \to Y\to 0$ be an exact sequence in which $X,Y$ are $G$-spaces and
$\Omega: \Delta \lop X$ with ambient space $\Sigma$.  If $\Omega$ is a $G$-centralizer then TFAE:
\begin{itemize}
\item[(a)] The quotient map admits a $G$-linear section $\mathcal L: \Delta\To X\oplus_\Omega Y$.
\item[(b)] There is a $G$-linear map $\ell: \Delta\To \Sigma$ such that $\Omega - \ell: \Delta \To X$.

\noindent If, moreover, $\Delta\subset \Dom(\Omega)$ then (a) and (b) are also equivalent to:

\item[(c)] $y \to (0, y)$ is a $G$-linear section $\Delta\To X\oplus_\Omega Y$ for the quotient map.
\end{itemize}\end{lemma}

\begin{proof} Just set $\mathcal L y= (\ell y, y)$.
\end{proof}
The reason why the hypothesis ``$\Omega$ is a $G$-centralizer" is needed is to be sure that $\lambda(g) = \begin{pmatrix} u(g) & 0 \\ 0 & v(g) \end{pmatrix}$ on $X \oplus_\Omega Y$ is a compatible bounded action on $X\oplus_\Omega Y$. To describe the general situation we first need to develop  a few ideas. The general version of Proposition \ref{baba} will be presented in Proposition \ref{BABA} and that of Lemma \ref{babal} in Lemma \ref{BABAL}. The fact that $G$-centralizers are, informally speaking, quasilinear maps having uniformly bounded commutators $[u(g), \Omega, v(g) ]$ suggests to consider with special attention the case $[u, \Omega, v] = 0$:

\adef A quasilinear map $\Omega: \Delta\lop X$ will be called $G$-equivariant if $[u(g),\Omega,v(g)]=0$ for every $g\in G$.\zdef

In particular, $G$-equivariant linear maps (operators) are the $G$-linear maps (operators) of Defi\-nition \ref{listofdef}. Since $G$-equivariant maps, as well as their bounded perturbations, are $G$-centralizers, it is natural to ask about the converse: Is a $G$-centralizer always a bounded pertubation of a $G$-equivariant map? And its ``linear" version:  is a linear $G$-centralizer always a bounded perturbation of a $G$-linear map?

Let us provide an optimal answer: yes when $G$ is an amenable group and $X$ is adequately complemented in its bidual. Recall that a Banach space is called an ultrasummand \cite{kaltnachr} when it is complemented in its bidual. We transplant this notion to $G$-Banach spaces. By $G$-projection we mean a $G$-operator which is a projection.

\begin{defin} A $G$-Banach space $X$ is a $G$-ultrasummand if there exists a $G$-projection $P:X^{**}\to X$. \end{defin}

Observe that if $X$ is a $G$-space then also $X^{*}$, hence $X^{**}$, is a $G$-space so that $X$ is a $G$-subspace of $X^{**}$, so the above makes sense. One has:

\begin{prop}\label{rrr}\label{222} Let $G$ be an amenable group and let $X,Y$ be $G$-spaces with $X$ a $G$-ultrasummand.
(a) Any (linear) $G$-centralizer $\Omega: Y \lop X$ is a bounded perturbation of a $G$-equivariant (linear) map. (b) A trivial $G$-centralizer $\Omega: Y \lop X$ is boundedly equivalent to a $G$-linear map. \end{prop}
\begin{proof} Proof of (a): since $G$ is amenable, there is a left invariant measure $\mu$ on it, and since $X$ is a $G$-ultrasummand there is a $G$-projection $P:X^{**}\to X$. We define the bounded map $B: Y \rightarrow X$
$$B y=P\left (\int_G (u(g^{-1})\Omega v(g)y -\Omega y)d\mu\right) $$
where we integrate in the weak* sense. If $h \in G$ then
\begin{eqnarray*}
B (v(h)y) &=& P\int_G \left(u(g^{-1})\Omega v(g)(v(h)y) -\Omega (v(h)y)\right)d\mu\\
&=&P\int_G (u(h)u(h^{-1}g^{-1}) \Omega v(gh)y -\Omega (v(h)y) ) d\mu\\
&=&P\int_G (u(h)u(h^{-1}g^{-1}) \Omega v(gh)y -u(h)\Omega y + u(h)\Omega y- \Omega (v(h)y) ) d\mu\\
&=&P\left(u(h) \int_G (u(h^{-1}g^{-1}) \Omega v(gh)y -\Omega y)d\mu  + \int_G (u(h)\Omega y- \Omega (v(h)y) ) d\mu\right)\\
&=& u(h) By + u(h)\Omega y -\Omega v(h)y\end{eqnarray*}
and therefore $[u(h),B,v(h)] = -[u(h),\Omega,v(h)]$, from where $[u(g), B+\Omega,  v(g)]=0$ for all $g \in G$. Namely, $B+\Omega$ is $G$-equivariant. The second part is clear: when $\Omega$ is linear, $B$ is also linear. Proof of (b): if $\Omega = B +L$ with $B$ bounded and $L$ linear, $L$ must also be a $G$-centralizer. Then apply (a). \end{proof}

Part (b) extends \cite[Lemma 1]{jesusjesusfelix}: a trivial $L_\infty$-centralizer is a bounded perturbation of a linear $L_\infty$-centralizer. As announced, the previous solution is optimal since the amenability condition is necessary. Let us put the counterexample in the proper context. As was proved by Day \cite{day} and Dixmier
\cite{dix}, a bounded representation of a countable amenable group on the Hilbert space is unitarizable, meaning that it is a unitary representation in some equivalent Hilbert norm.
Ehrenpreis and Mautner \cite{EM} provide a non-unitarizable bounded representation of a countable group on the Hilbert space.
What is now known as the Dixmier problem asks whether unitarizability of all bounded representations of a countable group characterizes amenability.

Regarding the non-amenable free group  $\F_\infty$ with countably infinitely many generators, Pytlic and Szwarc \cite{PS}, see also \cite{O,P}, showed the existence of a bounded, non-unitarizable representation of $\F_\infty$ on the sum $H \oplus H$ of two copies of the Hilbert space. The authors of \cite{FR2} used this example to investigate transitivity properties of bounded actions on the Hilbert space, and we now follow their lines with another perspective in mind. As in \cite{FR2} we extend the action of $\F_\infty$ to ${\rm Aut}(\mathbb T)$ where $\mathbb T$ is the Cayley graph of $\F_\infty$ with respect to its free generating set. Indeed, ${\rm Aut}(\mathbb T)$ acts in a natural way on $\ell_2(\mathbb T)$ as well as on
$\ell_\infty(\mathbb T)$ or $\ell_1(\mathbb T)$, by the left regular unitary representation $u$: $u(g)(x_t)_{t \in \mathbb  T}=(x(g^{-1} t))_{t \in \mathbb  T}$. Let $R:\ell_\infty(\mathbb T) \rightarrow \ell_\infty(\mathbb T)$ be $$R(e_t)= \mathrm{weak^*} - \sum_{                                                                                                                          s \in \F_\infty, \quad t<s, \quad|s|=|t|+1} e_s.$$

Since $[u(g),R]: \ell_2(\mathbb T) \rightarrow \ell_2(\mathbb T)$ has norm at most $2$ for all $g \in {\rm Aut}(\mathbb T)$ \cite{FR2}, $R$ is an ${\rm Aut}(\mathbb T)$-centralizer $\ell_2(\mathbb T) \lop \ell_2(\mathbb T)$, which is moreover trivial since it is linear (note that we have chosen $\Delta=\ell_2(T)$ and $\Sigma=\ell_\infty(T)$ here). We may  obtain another ${\rm Aut}(\mathbb T)$-centralizer through  the dual situation of the ``left shift" operator $L: \ell_1(\mathbb T) \to \ell_1(\mathbb T)$ defined as $L(e_t)=e_{\hat{t}}$ where $\hat{t}$ is the predecessor of $t$ along $\mathbb T$, and $L(e_\emptyset)=0$ (here we have chosen  $\Delta=\ell_1(\mathbb T)$
and $\Sigma=\ell_2(\mathbb T)$). Note that both $R$ and $L$ could also be defined as from $\ell_1(\mathbb T)$ to $\ell_\infty(\mathbb T)$, in which setting $L+R$ makes sense. Since $L+R$ commutes with every $g \in {\rm Aut}(\mathbb T)$, we have
$[u(g),L] = -[u(g),R]$
and so $L$ is also an ${\rm Aut}(\mathbb T)$-centralizer
$L: \ell_2(\mathbb T) \lop \ell_2(\mathbb T)$.
One has:

\begin{prop}\label{pysw} The linear ${\rm Aut}(\mathbb T)$-centralizer $R$  is not boundedly equivalent to a linear ${\rm Aut}(\mathbb T)$-equivariant map defined on the whole $\ell_2(\mathbb T)$. The linear ${\rm Aut(T)}$-centralizer $L$ is not boundedly equivalent  to a linear ${\rm Aut}(\mathbb T)$-equivariant map defined on $\Delta=\ell_1(\mathbb T)$.\end{prop}

\begin{proof}  Since $R$ takes values in $\ell_\infty(\mathbb T)$ but not in $\ell_2(\mathbb T)$, such a
 a linear  ${\rm Aut}(\mathbb T)$-equivariant map would have the same property. But it  is proved in \cite{FR2} that any linear (unbounded) map ${\rm Aut}(\mathbb T)$-equivariant map from $\ell_2(\mathbb T)$ to $\ell_\infty(\mathbb T)$ must be homothetic, and in particular it must take value in $\ell_2(\mathbb T)$. Regarding $L$, the linear equivariant map $\ell$ would have to be continuous from $\ell_1(\mathbb T)$ to  $\ell_2(\mathbb T)$. The dual map
would then be continuous from $\ell_2(\mathbb T)$ to $\ell_\infty(\mathbb T)$, and therefore would be homothetic, so $\ell$ itself would be homothetic. So $L$ would be $\|.\|_{\ell_2(\mathbb T)}-\|.\|_{\ell_2(\mathbb T)}$ bounded. This is false, since
for $x=\sum_{t \in N} e_t$, where $N$ is a family of $n$ elements of $\F_\infty$ of length $1$, we have $\|L(x)\|_2=\|n e_{\emptyset}\|_2=n$ while $\|x\|_2=\sqrt{n}.$ \end{proof}


\

We now study the general case, namely, sequences $0\to X \to X\oplus_\Omega Y \to Y\to 0$ in which there is a compatible action $\lambda$ on $X\oplus_\Omega Y$ but it is not ``diagonal". The first observation is that a compatible
action $\lambda$ has necessarily the form
$$\left(
    \begin{array}{cc}
      u(g) & d(g) \\
      0 & v(g) \\
    \end{array}
  \right)$$
with $d(g)$ a linear (not necessarily bounded, even when $\lambda(g)$ is) map from $\Delta$ to $\Sigma$. Observe that a compatible bounded nonlinear action on $X\oplus_\Omega Y$ always exists, and it is given by
$$\left(
    \begin{array}{cc}
      u(g) & -[u(g), \Omega, v(g)] \\
      0 & v(g) \\
    \end{array}
  \right).$$

This sets the key idea of how $d$ could be found: the map $g\to [u(g), \Omega, v(g)]$, that we will denote $[u, \Omega, v]$, is a (nonlinear) \emph{derivation} of $g \mapsto (u(g),v(g)$, in the sense that it is a map $d: G \to \Sigma^\Delta$ such that $d(gh)= u(g)d(h) + d(g)v(h)$. Of course that if $L$ is linear, then $[u, L, v]$ is linear. It could also occur that $\Omega$ and the actions $u,v$ are so well coordinated as to make
$[u, \Omega, v]$ linear: such is the case when $\Omega$ is the Kalton-Peck map, see Section \ref{actkothe}.

Derivations are of course fundamental for the study of unitarizability of bounded representations on the Hilbert space, such as the above representation of ${\rm Aut}(\T)$; we address the reader to Pisier's book \cite{P}. They also have been studied on direct sums of Banach spaces \cite{FR2} but, as far as we know, not on twisted sums. To perform such an study we must begin relaxing the requirement that $[u(g),\Omega,v(g)]$ is linear to ``being at uniform distance to a linear map", in the sense of the next definition:

\adef \label{derivation} Let $X, Y$ be $G$-spaces with respective actions $u$ and $v$. We say that $g \mapsto d(g)$ is a {\em linear derivation} of $(u, v)$ if for all $g \in G$, $d(g): \Delta \To \Sigma$ is a (possibly unbounded) linear map, and $d(gh)=u(g)d(h)+d(g)v(h)$ for all $g,h \in G$.
If, moreover, $\sup_{g \in G}\|[u(g),\Omega,v(g)]+d(g)\|<\infty$ then we will say that $d$ is an {\em $\Omega$-derivation} (of $(u,v)$) on $G$,
or that it is a derivation (of $(u,v)$) associated to $\Omega$.\zdef

We are ready to obtain the general version of Proposition \ref{baba}:

\begin{prop}\label{BABA}\label{compatibility} Let $\Omega: \Delta \lop X$ be a quasi-linear map between two $G$-spaces. TFAE:
\begin{itemize}
\item[(a)] $\lambda(g) = \begin{pmatrix} u(g) & d(g) \\ 0 &  v(g) \end{pmatrix}$ is a compatible bounded action of $G$ on $X \oplus_\Omega Y$.
\item[(b)] $g\to d(g)$ is a linear $\Omega$-derivation of $(u,v)$ on $G$.
\end{itemize}
\end{prop}
\begin{proof} The equality $\lambda(gh)=\lambda(g)\lambda(h)$ means
$$\begin{pmatrix} u(gh) & d(gh) \\ 0 &  v(gh) \end{pmatrix} = \begin{pmatrix} u(g) & d(g) \\ 0 &  v(g) \end{pmatrix}
\begin{pmatrix} u(h) & d(h) \\ 0 &  v(h) \end{pmatrix}= \begin{pmatrix} u(g)u(h) & u(g)d(h) + d(g)v(h) \\ 0 &  v(g)v(h) \end{pmatrix}.$$
The boundedness condition is a straightforward computation.\qedhere
\end{proof}

And, as promised, the general version of Lemma \ref{babal}.

\begin{lemma}\label{BABAL}\label{info}Let $\Omega, \Omega': \Delta \lop X$ be quasilinear maps between $G$-spaces $Y$ and $X$, with ambient space $\Sigma$, and let $L:\Delta\To \Sigma$ be a linear map. Then \begin{itemize}
\item[(a)] $ d(g)= - [u(g), L, v(g)]$ is an $L$-derivation.
\item[(b)] $\Omega$ is a $G$-centralizer if and only if $d=0$ is an $\Omega$-derivation. In particular, homogeneous bounded maps admit associated derivation $d=0$.
\item[(c)] If $d$ is an $\Omega$-derivation and $d\, '$ is an $\Omega'$-derivation then $d + d\,'$ is an $(\Omega + \Omega')$-derivation.
In particular, $\Omega+L$ is a $G$-centralizer if and only if $[u, L, v]$ is an $\Omega$-derivation.

\noindent If, moreover, $\Delta \subset \Dom(\Omega + L)$ then

\item[(d)] $\Omega + L$ is a $G$-centralizer if and only if $\mathcal L: \Delta\To X \oplus_\Omega Y$ given by $\mathcal L(y)= (- Ly, y)$ is a $G$-linear section for the quotient map $X \oplus_\Omega Y \rightarrow Y$.
\end{itemize}\end{lemma} To avoid confusion let us make clear that all derivations in this lemma are meant to be derivations of the given pair of representations $(u,v)$.

\begin{proof} (a) and (b) are clear. (c) is a simple consequence of the fact that $d + d\,'$ is linear and $[u,\Omega + \Omega',v]=
 [u,\Omega, v] + [u,\Omega',v]$. (d) is clearly the general version of Lemma \ref{babal} (c) with a couple of delicate points to check: that $(- Ly, y)\in  X \oplus_\Omega Y$, which is true when $y\in \Dom(\Omega + L)$ and that $\mathcal L$ is $G$-linear. To this end, observe that
$$\begin{pmatrix} u(g) & d(g) \\ 0 & v(g) \end{pmatrix} \begin{pmatrix} -Ly \\ y \end{pmatrix}=
\begin{pmatrix} -Lv(g)y \\ v(g)y \end{pmatrix}$$

is equivalent to $d(g) = [u(g), L, v(g)]$, which is equivalent to $\Omega + L$ is a $G$-centralizer since $d(\Omega + L)= d(\Omega) - [u,L,v]$.\end{proof}

The example around Proposition \ref{pysw} shows two essentially different bounded actions of ${\rm Aut}(\mathbb T)$ on $\ell_2(\mathbb T)\oplus \ell_2(\mathbb T)$: one is the unitary action $\begin{pmatrix} u(g) & 0 \\ 0 & u(g) \end{pmatrix}$ and the other is $\begin{pmatrix} u(g) &  [u(g),L]  \\ 0 & u(g) \end{pmatrix}$. By the above discussion,  this triangular action on $\ell_2(\T) \oplus \ell_2(\T)$ and the diagonal one on
$\ell_2(\mathbb T)\oplus_L \ell_2(\mathbb T)$ are ``the same".
Shifting the classical perspective,  we can therefore reformulate this construction as the remarkable fact that ${\rm Aut}(\mathbb T)$ with its diagonal action, is ``centralized" by two essentially different quasilinear maps: $0$ and $L$.

Thus, all pieces are on the board, except one: how to obtain a linear derivation of a quasilinear $G$-compatible map (assuming it exists)? The context of interpolation will provide some answers, and this is the content of the next section.

\section{Actions on interpolation scales}\label{actions} We now consider exact sequences of $G$-spaces generated by complex interpolation of a scale on which $G$ acts, in a way to be described. We refer to \cite{belo},\cite{CZ} or \cite{RW} (see also \cite{handbook} or \cite{CCFG} for specific details) for sounder information about the complex interpolation method for pairs and their associated differentials. An \emph{interpolation pair} $(X_0, X_1)$ is a pair of Banach spaces, both of them linearly and continuously contained in a larger Hausdorff topological vector space $\Sigma$, which can be
assumed  to be $\Sigma = X_0+X_1$ endowed with the norm
$\|x\|= \inf \{\|x_0\|_0 + \|x_1\|_1: x= x_0 + x_1\; x_j\in X_j \; \mathrm{for}\; j=0,1\}$.
The pair will be called {\em regular} if, additionally, the intersection space $X_0\cap X_1$ is dense in both $X_0$
and $X_1$. We denote by $\SSS$ the complex strip defined by $0<Re(z)<1$. A \emph{Calder\'on space} $\mathscr{C}$ is a certain Banach space of holomorphic functions $F:\overline{\SSS} \to X_0+X_1$ for which the evaluation maps $\delta_z:\mathscr{C}\to \Sigma$ are continuous. This forces the evaluation of the derivatives $\delta_z': \mathscr{C}\to \Sigma$ to be continuous too. The interpolation spaces are defined to be $X_z=\{x\in {\Sigma}: x = f(z) \text{ for some } f\in\mathscr C\}$ endowed with natural quotient norm. There are various possible choices for $\mathscr C$ and  we shall consider either $\mathscr C(\mathbb{S}, X_0+X_1)$, the classical Calder\'on space (see \cite{belo}), or ${\mathcal F}_z^{\infty}$, Daher's space (as in \cite{D} or \cite[Section 5]{CCFG}). If $B_z: X_z\to \mathscr{C}$ is a homogeneous bounded selection for the evaluation map, the \emph{differential map} of the process is $\Omega_z = \delta_z' B_z: X_z \to \Sigma$. This is a quasilinear map $\Omega_z: X_z\lop X_z$. Other choices of a selection $B_z$ lead to boundedly equivalent differentials $\Omega_z$.

An operator $\tau:\Sigma \to \Sigma$ is said to act {\em on the scale} defined by the interpolation pair $(X_0,X_1)$ if it is a bounded operator $X_i \to X_i$, $i=0,1$ \cite{CCFG}. A generalized Riesz-Thorin theorem \cite{RT} yields that $\tau$ is automatically bounded from $X_\theta\to X_\theta$ for all $0<\theta<1$, with an estimate
$\|\tau\|_{{\mathcal L}(X_\theta)} \leq \|\tau\|_{{\mathcal L}(X_0)}^{1-\theta} \|\tau\|_{{\mathcal L}(X_1)}^\theta.$

\begin{defin} Let $(X_0,X_1)$ be a complex interpolation pair. A semigroup $G$ acting on $\Sigma$ is said to {\em act on the scale} if $G$ acts boundedly on $X_i$ for $i=0,1$.\end{defin}

The interpolation estimate  above implies that $G$ also acts on $X_\theta$ for all $0<\theta<1$ and that if $G$ acts as an isometry group on the scale then it also acts as an isometry group on $X_\theta$, $0<\theta<1$, as well as on $\Sigma$ and $X_0 \cap X_1$. The same holds for semigroups of contractions. Moreover, $\mathcal C$ is a $G$-Banach space defined by the action $g^{\mathcal C} (f)(z) = g (f(z))$
with estimate (in the standard situations described above) $\|g^{\mathcal C}\|\leq \|g\|$. Note that the actions $u$ and $v$ in this setting are simply $u(g)=v(g)=g$. Where is our promised derivation? Here: 0.

\begin{prop}\label{Gcentralizer} If $G$ is a semigroup acting on the scale $(X_0, X_1)$ then $\Omega_\theta$ is a $G$-centralizer on $X_\theta$.
\end{prop}

\begin{proof} For $x \in X_\theta$ one has $g^{\mathcal C} \left(B_\theta x\right) -B_\theta (gx)\in \ker \delta_\theta$. Therefore
\begin{eqnarray*} \left \|(g\Omega_\theta-\Omega_\theta g)x\right \|_\theta &=& \|g\delta_\theta' B_\theta x - \delta_\theta' B_\theta (g x)\|_\theta \\
&=& \|\delta_\theta' \left (g^{\mathcal C}(B_\theta x) - B_\theta (g x)\right)\|_\theta\\
&\leq& \|\delta_\theta'|_{\ker \delta_\theta}\| \|g^{\mathcal C}(B_\theta x)  - B_\theta (g x) \|_{\mathcal C} \\
&\leq& \|\delta_\theta'|_{\ker \delta_\theta}\|\left( \|g^{\mathcal C}(B_\theta x)\|_{\mathcal C}+\|  B_\theta (g x) \|_{\mathcal C}\right) \\
&\leq& \|\delta_\theta'|_{\ker \delta_\theta}\| 2\|B_\theta\|\|g\|\; \|x\|_\theta.\qedhere\end{eqnarray*}
\end{proof}

Proposition \ref{Gcentralizer} admits an isometric version that we formulate now. A regular interpolation pair is {\em optimal} if for every $0<\theta<1$, every point in $X_\theta$ admits a unique minimal function in ${\mathcal F}_\theta^{\infty}$, see \cite[Def. 5.7]{CCFG}. Daher proved in
\cite[Prop. 3]{D} that a regular pair of reflexive spaces with $X_0$ strictly convex is optimal.

\begin{corollary}\label{Gequivariant} Let $(X_0,X_1)$ be an optimal  interpolation pair. Then
$\Omega_\theta$ is equivariant with respect to the semigroup of contractions on the scale which act as isometric embeddings on $X_\theta$. In particular, $\Omega_\theta$ is equivariant with respect to the group of isometries acting on the scale.
\end{corollary}

\begin{proof} The map $\Omega_\theta$ is uniquely defined now since $(B_\theta x)(\theta)=x$ and $\|B_\theta x\| = \|x\|_\theta$.
If $g$ is a contraction on the scale, then $g^{\mathcal C}$ also acts as a contraction on the  space ${\mathcal F}_\theta^\infty$.
Since $\|g^{\mathcal C}B_\theta x\| \leq \|B_\theta x\|=\|x\|_\theta=\|gx\|_\theta$ if $g$ is also an isometric embedding on $X_\theta$, and since $g^{\mathcal C}(B_\theta x)(\theta)=gx$, we deduce that $g^{\mathcal C}(B_\theta) x = B_\theta( gx).$ Derivating in $\theta$ implies that $\Omega_\theta g=g\Omega_\theta$. \end{proof}

It is a bit disappointing that a zero derivative is all we got. There is a reason for that: the action of $G$ on the scale $(X_z)$ is constant: $u_z(g)=g, \forall z$. To amend this, consider for each $z$ a bounded action $u_z: G \to \mathfrak L(\Sigma)$
such that $u_z(g)|_{X_z}: X_z\to X_z$.
Recall that a function $f: \mathbb S\to \mathfrak L(Y, X)$ is analytic if for every $y\in Y$ the function $z\to f(z)(y)$ is analytic (here $X$ must be interpreted as a part of $\mathfrak L(X^*, \mathbb C)$).

\adef The family of actions $\mathfrak u = (u_z)$ is analytic if for each $g\in G$ the function $z\to u_z(g) \in \mathfrak L(\Sigma, \Sigma)$ is analytic.\zdef

Assume one has a semigroup $G$ with an action $u$ on $X_\theta$ (only for that fixed $\theta$!) but in such a way that $u=u_\theta$ for some analytic family $(u_z)$ of actions. Thus, the compatible action of $G$ on $X_\theta \oplus_{\Omega_\theta} X_\theta$ will no longer be $\left(
                                                         \begin{array}{cc}
                                                           u(g) & 0 \\
                                                           0 & u(g) \\
                                                         \end{array}
                                                       \right)$. Since $\left(
                                                         \begin{array}{cc}
                                                           u(g) & -[u(g), \Omega_\theta] \\
                                                           0 & u(g) \\
                                                         \end{array}
                                                       \right)$ is a compatible, but nonlinear, bounded action, what we need is to find linear bounded perturbations of $[u(g), \Omega_\theta]$. We use here some ideas of Carro \cite{Carro}:

\begin{lemma}\label{carro} Let $\mathfrak u= (u_z)_{z\in \mathbb S}$ be an analytic family of actions of $G$ on the spaces of the scale $(X_z)_{z\in \mathbb S}$ generated by a regular pair $(X_0, X_1)$. The following map is bounded $$\left [u_\theta(g), \Omega_\theta\right ] + \frac{d u_z (g)}{dz}|_\theta : X_\theta \To X_\theta.$$\end{lemma}
\begin{proof} The key observation is that for $x\in X_\theta$ the function $u_z(g)\left( B_\theta x\right) (z) - B_\theta(u_\theta(g) x) (z)\in \ker \delta_\theta$ which implies that its derivative at $\theta$ must be in $X_\theta$. It only remains to compute
\begin{eqnarray*} \left(u_z(g)B_\theta x(z) - B_\theta(u_\theta(g) x) (z)\right)'(\theta) &=&
u_\theta(g) \Omega_\theta(x)  + \frac{d u_z(g)(x)}{dz} |_\theta - \Omega_\theta(u_\theta(g)x)\\ &=& [u_\theta(g), \Omega_\theta](x) + \frac{d u_z(g)(x)}{dz} |_\theta . \qedhere \end{eqnarray*}
\end{proof}
This means that $\lambda(g) =  \left(\begin{array}{cc} u_\theta(g) &   \frac{d\mathfrak u(g)}{dz}|_\theta \\
                                                           0 & u_\theta(g) \\
                                                         \end{array}
                                                       \right):  X_\theta \oplus_{\Omega_\theta} X_\theta \To X_\theta \oplus_{\Omega_\theta} X_\theta$ is a bounded operator. To obtain a bounded action we need that $\sup_g \|\lambda(g)\|<+\infty$.
Since \begin{eqnarray*}
\left \| \left(\begin{array}{cc}u_\theta(g) &  -[u_\theta(g), \Omega_\theta] \\
                                                           0 & u_\theta(g) \\
                                                         \end{array}
                                                       \right) -
                                                       \left(\begin{array}{cc} u_\theta(g) &  \frac{d u_z(g)}{dz}|_\theta \\
                                                           0 & u_\theta(g) \\
                                                         \end{array}
                                                       \right)\right\|&=&
                                                       \left\| \left(\begin{array}{cc} 0 &  [u_\theta(g), \Omega_\theta]+ \frac{d u_z(g)}{dz}|_\theta \\
                                                           0 & 0 \\
                                                         \end{array}
                                                       \right)\right\|\\
                                                       &=& \|[u_\theta(g), \Omega_\theta] + \frac{d u_z(g)}{dz}|_\theta\| \end{eqnarray*}
and $ \left(\begin{array}{cc} u_\theta(g) &  -[u_\theta(g), \Omega_\theta] \\
                                                           0 & u_\theta(g) \\
                                                         \end{array}
                                                       \right)$ is uniformly bounded, what we need is
$$\sup_{g\in G}\|[u_\theta(g), \Omega_\theta] + \frac{d\mathfrak u(g)}{dz}|_\theta\|<\infty.$$ Now, if $\mathfrak u = (u_z)$ is an analytic family of actions and we are using the standard  Calder\'on space $\mathcal C = \mathcal C(X_0, X_1)$ endowed with the norm $\|h\|_\infty:=\sup \{\|h(it)\|, \|h(1+it)\|: t\in \R\}$ then we set $\gamma(\mathfrak u) = \sup_{g\in G}\sup_{t\in \R} \{ \|u_{it}(g)\|, \|u_{1+it}(g)\| \}$ (since $X_{z+it}=X_z$ for real $t$, we may assume that $u_{z+it}(g)=u_{z}(g)$). We have:
\begin{eqnarray*}
\|u_z(g)\left(B_\theta x)  (z) \right)\|_{\mathcal C} &=& \sup_{t\in \R} \{ \|u_{it}(g)B_\theta x (it)\|, \|u_{1+it}(g) B_\theta x (1+it)\|\}\\
&\leq& \gamma(\mathfrak u) \| B_\theta \|\|x\|.
\end{eqnarray*}
The interpolation inequality yields $\|u_\theta(g)\|\leq \gamma(\mathfrak u)$ and thus one has \begin{eqnarray*}
\left \| [u_\theta(g), \Omega_\theta] + \frac{d u_z(g)}{dz} |_\theta\right\|_\theta&=&
\|\left(u_z(g)B_\theta x(z) - B_\theta(u_\theta(g) x) (z)\right)'(\theta)\|_\theta\\
&\leq& \|\delta_\theta': \ker \delta_\theta \to X_\theta\| \|u_z(g)B_\theta x  - B_\theta(u_\theta(g) x)\|_{\mathscr C}\\
&\leq& \|\delta_\theta': \ker \delta_\theta \to X_\theta\| \left( \|u_z(g)B_\theta x\|_{\mathcal C} +  \|B_\theta\| \|u_\theta(g)\| \|x\|_\theta\right)\\
&\leq& 2\|\delta_\theta': \ker \delta_\theta \to X_\theta\| \gamma (\mathfrak u)\|B_\theta\|\|x\|_\theta.\end{eqnarray*}

All this yields,

\begin{theorem}\label{differential} Let  $(X_0, X_1)$ be a regular pair. Let $\mathfrak u$ be an analytic family of actions of $G$ on the scale $(X_z)_{z\in \mathbb S}$ such that $\gamma(\mathfrak u)<\infty$. Then
$$\lambda(g) =
                      \left(                                   \begin{array}{cc}
                                                           u_\theta(g) &   \frac{d\mathfrak u_z(g)}{dz}|_\theta \\
                                                           0 & u_\theta(g) \\
                                                         \end{array}
                                                       \right)$$
                                                       is a compatible action of $G$ on $X_\theta \oplus_{\Omega_\theta} X_\theta$, or equivalently,
$g \mapsto \frac{d\mathfrak u_z(g)}{dz}|_\theta$ is an $\Omega_\theta$-derivation of $(u_\theta,u_\theta)$. \end{theorem}

It is certainly satisfying that the term ``derivation" agrees here both with the classical meaning and with Definition \ref{derivation}!
Let us provide the first of a series of natural applications of these results. Let $(X_0,X_1)$ be an optimal interpolation pair of uniformly convex and uniformly smooth spaces and $0<\theta<1$. We claim that the semigroup of contractions on $X_\theta$ is compatible with $\Omega_\theta$: To prove it, we set a bounded action $\begin{pmatrix} g & d(g) \\ 0 & g \end{pmatrix}$ on $X_\theta \oplus_{\Omega_\theta} X_\theta$ given by $d(\phi \otimes x)=\Omega_\theta^*(\phi) \otimes x + \phi \otimes \Omega_\theta(x), $ whenever $x \in X_\theta, \phi \in X_\theta^*$, and where $\Omega_\theta^*$ denotes the quasi-linear map induced on $X_\theta^*$ by the identity $X_\theta^*=(X_0^*,X_1^*)_\theta$. If $g(y)=\langle \phi,y \rangle x$ with $\phi, x$ is a contraction of rank $1$ with norm at most $1$ on $X_\theta$, let
$F$ be an extremal for $x\in X_\theta$ and $G$ an extremal for $\phi\in (X_0^*,X_1^*)_\theta$). Then $\mathfrak g_z(y)=\langle G(z),y \rangle F(z)$  defines an analytic family of contractions of rank $1$ on the scale $(X_z)$ such that $\mathfrak g_\theta=g$. Now apply Lemma \ref{carro} after calculating
$$\frac{d \mathfrak  g(y)}{dz}|_\theta = \langle \Omega_\theta^*(\phi), y\rangle x + \langle \phi,y\rangle \Omega_\theta(x).$$


Further applications will be given in Sections \ref{actkothe} and \ref{opz2}.

\section{Actions on K\"othe spaces}\label{kothe}

When working with K\"othe spaces with base measure space $S$, the ambient $\Sigma$ is usually chosen as the space $L_0(S)$ of measurable functions on $S$, and $\Delta$ as $Y$ itself. A K\"othe space is a vector subspace $K$ of $L_0(S)$ endowed with a norm such that $f\in K$ and $|g|\leq |f| $ then $g\in K$ and $\|g\|\leq \|f\|$.  A r.i. K\"othe space over $[0,1]$ is a K\"othe space $X$ such that $f\in X \Rightarrow f\sigma\in X$ for every measure preserving map $\sigma: [0,1]\to [0,1]$. K\"othe spaces are usually considered in their $L_\infty$-module and $L_\infty$-centralizer structures.
The notion of $L_\infty$-centralizer can be subsumed in our notion of $G$-centralizer when the spaces are $B$-convex (e.g. superreflexive), which is actually the most interesting situation regarding interpolation. Indeed, if $\mathcal U$ will denote the group of units of $L_\infty(\mu)$, i.e. of unimodular functions in $L_0(S)$ then

\begin{prop}\label{UL} Let $\Omega: Y \lop X$ be a quasilinear map, with $X,Y$ B-convex K\"othe spaces. Then $\Omega$ is an $\mathcal U$-centralizer if and only if it is an $L_\infty$-centralizer.\end{prop}
\begin{proof} The set of convex combinations of units is dense in the ball of $L_\infty$. If $\Omega$ is a $\mathcal U$-centralizer then,
 by Proposition \ref{1G}, it is a $\ell_1(\mathcal U)$-centralizer. We can use density plus $z$-linear character of $\Omega$ to conclude that $\|\Omega a-a \Omega\| \leq K$ for all $a$ in the unit ball of $L_\infty$. \end{proof}

$\mathcal U$-actions on K\"othe spaces have a somehow ``rigid" nature, whose paradigm is Kalton's stability theorem \cite{K}: the endpoint spaces of an interpolation scale of uniformly convex K\"othe spaces $X_0, X_1$ are uniquely determined, up to equivalence of norms, by the pair formed by the space $X_\theta$ and the differential $\Omega_\theta$, $0<\theta<1$. We additionally have:

\begin{theorem}\label{Geq-centralizer}
Let $(X_0, X_1)$ be an interpolation pair of superreflexive K\"othe spaces on a measure space $S$. Let $G$ be a group containing the group of units $\mathcal U(S)$, acting boundely on $X_\theta$ and acting on $\Sigma$. TFAE:
\begin{itemize}
\item[(a)] $\Omega_\theta$  is a $G$-centralizer.
\item[(b)] $G$ acts on the scale.
\end{itemize}
\end{theorem}\begin{proof} One implication is Proposition \ref{Gcentralizer}. Assume that $\Omega_\theta$ is a $G$-centralizer. For $g \in G$, and $i=0,1$
let $g^{-1}X_i \subset \Sigma$ be endowed with the complete norm $\|x\|^{g}_i=\|gx\|_i$. Form the new Calder\'on space $\mathcal C( g^{-1}X_0, g^{-1}X_1)$ and define an isomorphism $g^{\mathcal C}: \mathcal C(g^{-1}X_0, g^{-1}X_1)\to \mathcal C(X_0, X_1)$ in the form
$g^{\mathcal C}(h)(z)= gh(z)$. This yields $(g^{-1}X_0, g^{-1}X_1)_\theta= g^{-1}X_\theta =X_\theta$, with norm $\|x\|_\theta^{g}=\|gx\|_{X_\theta}$, which is equivalent to $\|.\|_\theta$ with a uniform constant independent of $g$. Since
$G: (g^{-1}X_0, g^{-1}X_1)_\theta \To \mathcal C(g^{-1}X_0, g^{-1}X_1)$ given by
$G(x) =g^{-1}B_\theta(gx)$ is a $(C\sup_{g \in G}\|g\|_\theta^2)$-extremal, we get the differential
$$\mhor_\theta(x) =(G(x))^{\prime}(\theta)= g^{-1}(B_\theta(gx))^{\prime}(\theta)=g^{-1}\Omega_\theta (gx)$$

Since $\Omega_\theta$ is a $G$-centralizer, $\mhor_\theta$ is boundedly equivalent to $\Omega_\theta$, with a uniform constant on $g$. Since $G$ contains the group $\mathcal U$ of units, $\Omega_\theta$ and $\mhor_\theta$ are $L_\infty$-centralizers. Kalton's stability theorem
will imply that the norms $\|.\|_i$ and $\|.\|_i^{g}$ are equivalent, with a constant independent of $g \in G$, as soon as we
amend in the next Lemma the required amalgamation. In conclusion, that $G$ acts on the scale.
\end{proof}

We will need to simultaneously consider differentials in various scales, so we will denote $\Omega^W$ the differential generated by $W:=(W_0, W_1)$.

\begin{lemma}\label{quantK} There exists a function $K(\cdot)$ such that whenever $X:=(X_0, X_1)$ and $Y:=(Y_0, Y_1)$ are interpolation pairs of superreflexive K\"othe spaces on the same measure space such that $Y_\theta=X_\theta$, with $C$-equivalence of norms, and $\Omega_\theta^X$ and $\Omega_\theta^Y$ are $C$-boundedly equivalent then the norms $\|\cdot\|_{X_i}$ and $\|\cdot\|_{Y_i}$ are $K(C)$-equivalent for $i=0,1$.\end{lemma}

\begin{proof} Otherwise, pick $C$ and couples $(X_0^n,Y_0^n)$, $(Y_0^n,Y_1^n)$ for which the conclusion of the theorem does not hold for $C$ and $K(n)=n$. The pairs $\ell_2(X_i^n)$ and $\ell_2(Y_i^n)$ generate $C$-equivalent interpolation spaces with $C$-boundedly equivalent differentials while their norms are are not equivalent, in contradiction with Kalton's theorem \cite{K} (in the version presented in \cite[Thm. 3.4]{CCFG}).\end{proof}

\section{Actions on Kalton-Peck spaces}\label{actkothe}

Differentials obtained from complex interpolation of pairs $(X_0, X_1)$ of two K\"othe spaces on the same base measure space are $L_\infty$-centralizers. The differential generated by the interpolation pair $(L_\infty(\mu), L_1(\mu))$ deserves special attention. As it is well-known $(L_\infty(\mu), L_1(\mu))_{1/p} = L_p(\mu)$; and if one picks positive normalized $f$ then $B(f)(z)=f^{pz}$ is an  extremal and thus for $\theta=1/p$ one gets $\Omega_\theta(f)= B(f)^{\prime}(\theta)= p  f \log(f)$.
In what follows we will call Kalton-Peck map on $L_p$ the $L_\infty$-centralizer $\KP: L_p \lop L_p$ defined by $\KP(f) = p f  \log \frac{f}{\|f\|}$ (the $p$ is important for duality issues). The twisted sum space $Z_p(\mu) =L_p(\mu) \oplus_{\KP} L_p(\mu)$ will be called the Kalton-Peck space. Especially interesting is the case $L_\infty(\mu)= \ell_\infty$ since Banach spaces with unconditional basis are $\ell_\infty$-modules.

Fix $1<p<\infty$ and let us think now about compatible $\ell_\infty$-actions on the Kalton-Peck space $Z_p$. The Kalton-Peck map has a peculiarity: if $w = (w_n)$ is an infinite sequence of successive normalized blocks in $\ell_p$ and $\tau_w :\ell_p\to \ell_p$ is the operator
$\tau_w(x) = \sum x_nw_n = w\cdot x$ then the commutator $[\tau_w, \KP]$ is linear:
$$p^{-1}[\tau_w, \KP](x) = w \cdot x \log x  - w\cdot x \log(w\cdot x) =  w\cdot x\log x  - w\cdot x (\log w + \log x) = x \cdot w \log w$$
Therefore, if we consider the semigroup $\texttt{BC}_p$ of the block contractions above on $\ell_p$ then we get:

\begin{lemma} There is a compatible bounded action of $\texttt{BC}_p$ on $Z_p$ given by
$$\left(
  \begin{array}{cc}
    \tau_w &  \tau_{\KP w} \\
    0 & \tau_w \\
  \end{array}
\right)$$
\end{lemma}

These operators were introduced by Kalton \cite{Ksymplectic} in the case $p=2$ to obtain isometric complemented copies of $Z_2$ inside $Z_2$. In the next section we will generalize these results.

\section{Actions on Rochberg spaces}\label{opz2}

We refer to \cite{cck} for possible unexplained definitions or facts. Given an interpolation pair $(X_0, X_1)$, the $n^{th}$ Rochberg space  $\mathcal R^n_z$ is defined to be the space $\mathcal R^n_z = \{(\frac{f^{(n-1)}(z)}{(n-1)!}, \dots, f'(z), f(z)): f\in \mathcal F\}$ endowed with its natural quotient norm. It is clear that $\mathcal R^0_z =X_z$ and $\mathcal R^1_z$ is isomorphic to $X_z \oplus_{\Omega_z} X_z$. It was shown in \cite{cck} that Rochberg spaces are connected forming natural exact sequences$$\begin{CD}
0@>>> \mathcal R^m_z @>>> \mathcal R^{n+m}_z @>>> \mathcal R^n @>>> 0
\end{CD}$$
generated by quasilinear maps $\Omega_z^{n,m}:\mathcal R_z^n \lop \mathcal R_z^m$ with ambient space $\Sigma^m$. We are especially interested in the maps
$$\Omega_\theta^{1,m}(x) = (\Omega_\theta^{(n-1)}(x), \Omega_\theta^{(n-2)}(x),  \dots, \Omega_\theta^{(2)}(x), \Omega_\theta^{(1)}(x))$$
with $\Omega_\theta^{(k)}(x) = \frac{d^k}{dz^k}B_\theta(x)|_{\theta}$. The sequences are entwined in natural commutative diagrams
\begin{equation}\label{psidiagram}
\begin{CD}
\mathcal R^k_z   @= \mathcal R^k_z \\
@VVV @VVV\\
\mathcal R^m_z   @>>> \mathcal R^{n+m}_z @>>> \mathcal R^n_z\\
@VVV @VVV @|\\
\mathcal R^{m-k}_z   @>>> \mathcal R^{n+m-k}_z   @>>> \mathcal R^n_z \\
\end{CD}\end{equation}
with short exact rows and columns (even if we omitted the $0$'s). We focus now on the situation in which the spaces of the scale have a common unconditional basis $(e_n)$ and an additional property. For $X$ with basis
$(e_n)$  we will call property $(W)$ the fact that for each normalized block sequence $w= (w_n)$ of $X$, the map $\tau_w: x\To w\cdot x$ is an operator of norm at most $1$ (equivalently, $\|\sum \lambda_n w_n\|\leq \|\sum\lambda_n e_n\|$); and that the maps $\tau_w$ form a semigroup for composition. Identifying $w$ with $\tau_w$, this allows us to see the set  of normalized block sequences $w=(w_n)$ on $X$ as a semigroup $\texttt{Block}_X$ acting on $X$.
Assume the spaces of the scale have property (W), and
given $\theta$, an analytic family of actions of $\texttt{Block}_{X_\theta}$ can be defined as follows: let $B_\theta$ be a homogeneous optimal selector for the evaluation map $\delta_\theta: \mathcal F \to X_\theta$ with the property that $\supp B_\theta(x)(z)\subset \supp x$ for each finitely supported $x$. This makes that for $w\in \texttt{Block}_{X_\theta}$ and all $z$ one has $B_\theta(w)(z)\in \texttt{Block}_{X_z}$. We define the following analytic family of actions $w_z(x)= B_\theta(w)(z)\cdot x$ and note that $w_\theta(x)=w.x$. Therefore
$\frac{dw_z(x)}{dz}|_\theta =   \frac{d}{dz} ( B_\theta(w)(z)\cdot x )|_\theta= \Omega_\theta(w)\cdot x$ and thus, by Theorem \ref{differential}, there is an action on $\mathcal R_z^2$ given by$\lambda_2(w) =
                      \left(                                   \begin{array}{cc}
                                                           w  &   \Omega_\theta(w)  \\
                                                           0 & w \\
                                                         \end{array}
                                                       \right)$ in accordance with the previous result for Kalton-Peck maps. The result can now be iterated for higher order Rochberg spaces to obtain:

\begin{prop} Let $(X_0, X_1)$ be an optimal interpolation pair of spaces such that $X_z$ has property (W) for each $z$. For fixed $\theta$ there is a bounded action of
the semigroup $\texttt{Block}_{X_\theta}$ of normalized block sequences of $X_\theta$  on $\mathcal R_z^n$ given by

$$\left(
    \begin{array}{ccccc}
      w & \Omega_{\theta}^{(1)}(w) & \Omega_\theta^{(2)}(w) & \dots & \Omega_\theta^{(n-1)}(w) \\
      0 & w & \Omega_\theta^{(1)}(w)  & \Omega_\theta^{(2)}(w) & \dots \\
      0 & 0 & w & \Omega_\theta^{(1)}(w)  & \Omega_\theta^{(2)}(w) \\
      \dots & 0 & 0 & w & \Omega_\theta^{(1)}(w) \\
      0 & \dots  & 0 & 0 & w \\
    \end{array}
  \right).$$
\end{prop}

This applies in particular to the scale of $\ell_p$ spaces, and provides new operators on Rochberg spaces that can be used to get insight on their properties. This research will appear elsewhere.

\section{Actions and almost transitivity}\label{almost} An isometric action $u$ of a group $G$ on a space $X$ is said to be {\em almost transitive} if the orbit $u(G).x$ is dense in $S_X$ for some (and therefore for all) $x \in S_X$, \cite{pelczynski}.
A bounded action $u$ of $G$  on $X$  is said to be almost transitive  if there is some $u(G)$-invariant renorming of $X$ for which the isometric action $u$ is almost transitive. The definition is independent of the choice of an invariant renorming, since all these renormings are multiple of each other \cite{cowie}.

\begin{prop}\label{translemma} Assume  $\Omega: Y \lop X$ is a $G$-centralizer and that $G$ acts almost transitively on $Y$. If
$\Dom\Omega \neq 0$ then $\Omega$ is bounded. \end{prop}

\begin{proof} The objective is to show that $\Dom\Omega \neq 0 \Rightarrow \Dom\Omega =Y$. Pick $y \in S_Y$ generating a dense orbit with respect to an equivalent invariant norm for the action $v$ and use also a norm on $X$ invariant under the action $u$. Then $\|\Omega(v(g)y) -u(g)\Omega y\|_X=\|[u(g),\Omega,v(g)]y \|_X \leq C$ therefore $\|\Omega v(g)y\|_Y \leq C+\|\Omega y\|_X,$ so $\Omega$ is bounded on a dense subset of $S_Y$.\end{proof}

$L_\infty$-centralizers acting on K\"othe spaces have dense dominion \cite{felix}. Consequently:

\begin{prop} Let $(X_0,X_1)$ be an interpolation pair with a common K\"othe space structure and let $0<\theta<1$. If $\Omega_\theta$ is unbounded then no bounded group acting on the scale can act almost transitively on $X_\theta$.\end{prop}

\begin{proof} If an almost transitive group $G$ acts boundedly on the scale, $\Omega_\theta$ would be a $G$-centralizer by Proposition \ref{Gcentralizer}, therefore would be bounded by Proposition \ref{translemma}. \end{proof}

Recall from  \cite{watbled} (see also \cite[Propositions 6.1 and 6.2]{CFG}) that if $X$ is a space with a shrinking basis then $(X,\overline{X}^*)_{1/2}$ is a Hilbert space. Therefore, if $X$ is either (a) a supereflexive K\"othe space on a measure space $S$ different from $L_2(S)$, or (b) a space with a shrinking basis such that the differential $\Omega_{1/2}$ generated at $(X,\overline{X}^*)_{1/2}$ is unbounded then no almost transitive bounded group of automorphisms on the Hilbert space $H$ can act boundedly on the scale, i.e. it cannot induce a bounded action on both $X$ and $\overline{X}^*$.\\

On the other hand, $\Dom \Omega=\{0\}$ is perfectly possible, and even $L_\infty$-centralizers acting between mere $L_\infty$-modules can have
trivial domain: the map $\Omega: \ell_2 \lop \ell_2 $ with ambient space $Z_2$ given by $\Omega(x) = (\KP x, x)$ is an $L_\infty$-centralizer since
\begin{eqnarray*} \|\Omega(\xi x) - \xi \Omega(x)\|_{Z_2} &=& \|(\KP (\xi x), \xi x) - \xi (\KP x, x) \|_{Z_2}\\
&=& \|(\KP (\xi x) - \xi \KP x, 0)\|_{Z_2}\\
&=& \|\KP (\xi x) - \xi \KP x)\|_2\\
&=& G(\KP)\|\xi\|_\infty \|x\|_2\end{eqnarray*}
with domain $\Dom(\Omega)  = \{x\in \ell_2: (\KP (x), x) \in \ell_2\} =\{0\}$. It also occurs in noncommutative contexts (see \cite[5.2]{noncom}), where the connection between nontrivial domain and almost-transitive action has been observed in \cite[p.140]{noncom}: ``One may wonder if [...] there is a real obstruction to have centralizers with nontrivial domain". Cabello yields then Example 5.2: where the almost transitive action of states implies that centralizers with nonzero domain are bounded.\\

$\bigstar$ Consider the group ${\rm Isom}(L_p)$ of isometries of $L_p(0,1)$, $p \neq 2$. One has:

\begin{prop}\label{thekp}$\;$
\begin{itemize}
\item $\KP$ is compatible with the natural action of ${\rm Isom}(L_p)$ on $L_p$.
\item $\KP$ is not an ${\rm Isom}(L_p)$-centralizer.
\item $\KP$ is a linear perturbation of an ${\rm Isom}(L_p)$ centralizer with trivial domain.
\end{itemize}\end{prop}
\begin{proof} To  show that $\KP$ is compatible with the action of ${\rm Isom}(L_p)$, observe that the elements of ${\rm Isom}(L_p)$ have the form $T(f)(s)=\varepsilon(s) w(s)^{1/p} (f \circ \phi)(s),$ where $\varepsilon$ is a unimodular map, $\phi$ a Borel isomorphism of $[0,1]$ and $w$ the Radon-Nikodym derivative of $\phi$ (by the Banach-Lamperti's formula \cite[Chapter 3]{isometries}). It follows in particular that $T(hf)=(h  \circ \phi). f$ whenever $h\in L_\infty(0,1)$. We show that, once again, $[\KP,T]$ is linear: if $f$ is a simple function of norm $1$, we have
\begin{eqnarray*} \frac{1}{p}[\KP,T]f &=& (Tf) \log|Tf| -T(f \log f)=(Tf) \log|Tf| -(Tf)\log(|f \circ \phi|)\\
&=& (Tf) \log \frac{|Tf|}{|f \circ \phi|} =\frac{1}{p}\log(w) (Tf). \end{eqnarray*}

An alternative form of finding this compatible action is considering the analytic family of actions $T_z(f)(t)=\varepsilon(t) w(t)^{z} (f \circ \phi)(t)$ to get $\begin{pmatrix} T_\theta & \frac{dT_z}{dz}(\theta) \\ 0 & T_\theta \end{pmatrix}.$

To prove the second part note that the group ${\rm Isom}(L_p)$ contains the units of $L_\infty$ and acts, linearly, on $L_0$ thus, from Theorem \ref{Geq-centralizer} we get that $\KP$ is an ${\rm Isom}(L_p)$-centralizer if and only if ${\rm Isom}(L_p)$ acts boundedly on the scale of $L_p$-spaces. Since
${\rm Isom}(L_p)$ acts almost transitively on $L_p$, see  \cite{isometries}, and $\KP$ has dense domain, we get from Proposition \ref{translemma} that if if ${\rm Isom}(L_p)$ acts boundedly on the scale of $L_p$-spaces then $\KP$ must be bounded on $L_p$, something it is not.

We know from Lemma \ref{perturlin} that there is a linear map $\mathcal L$ such that  $\KP + \mathcal L$ is an ${\rm Isom}(L_p)$-centralizer, precisely,
$\KP + \mathcal L: L_p\To L_p\oplus_{\KP} L_p$ given by $(\KP + \mathcal L)(y) = (\KP y, y)$. This map is an ${\rm Isom}(L_p)$-centralizer (Proposition \ref{fishy}) with trivial domain:
\begin{eqnarray*}
\Dom(\KP + \mathcal L)&=& \{y\in L_p: [((\KP y, y), 0)]\in \xi X\}\\
&=& \{y\in L_p: \exists x\in X: [((\KP y, y), 0)]=((x,0), 0)]\}\\
&=& \{y\in L_p: \exists x, x' \in X : ((\KP y - x , y), 0) = ((x',0), -x')\}\\
&=& \{ 0 \}.\qedhere\end{eqnarray*}
\end{proof}

We can provide additional information about this strange phenomenon; to ease notation we will call $G={\rm Isom}(L_p)$.

\begin{lemma} Let $L: \Delta\To \Sigma$ be a linear map such that $\KP + L$ is a $G$-centralizer. If $\Delta$ is a dense $G$-invariant subspace of $\Dom(\KP)$ then $\Delta \cap \Dom(L)=\{0\}$.\end{lemma}
\begin{proof} If $y \in \Delta \cap \Dom(L)$ then $(\KP + L)y \in L_p$. Since $g(\KP+L)-(\KP+ L)g$ is bounded, then $(\KP+L)z$ belongs to $L_p$ for all $z$ in the $G$-orbit of $y$; and since $\KP z \in L_p$ because $\Delta$ is $G$-invariant, we deduce that $Lz \in L_p$ on the $G$-orbit of $y$.
Let $\Delta' = \mathrm{span} (Gy) \subset \Dom \KP\cap \Dom L$. By quasi-transitivity of $L_p$, $Gy$ is dense on the unit sphere and $z\to (Lz, z)$ is a $G$-linear lifting for the quotient map $L_p\oplus_{\KP} L_p$ on $\Delta'$. Since
$\|(Lz,z)\|=\|(Lv(g)y, v(g)y)\|= \|\lambda(g) y\|\leq C\|y\|$ we actually have a linear bounded lifting on a dense subspace, and $\KP$ should be trivial, which is not.\end{proof}

\noindent In the particular case above, the result follows from $\Dom(\KP) \cap \Dom(\mathcal L) \subset \Dom(\KP + \mathcal L)=\{0\}$.\\

$\bigstar$ The case of the group ${\rm Isom}_{\rm disj}(L_2)$ of isometries of $L_2$ preserving disjointness is analogous: $\KP$ is compatible with the action of ${\rm Isom}_{\rm disj}(L_2)$, it is not an ${\rm Isom}_{\rm disj}(L_2)$-centralizer but it is a linear perturbation of an ${\rm Isom}_{\rm disj}(L_2)$-centralizer.\\

$\bigstar$ An even more stunning situation appears when considering the unitary group ${\rm Isom}(L_2)$

\begin{prop}\label{not2} $\KP$ is not compatible with the natural action of ${\rm Isom}(L_2)$ on $L_2$.\end{prop}

\begin{proof} Our starting point is the fact proved in \cite{CCFM} that some complex structure on $H$ does not extend to a complex structure on $Z_2$. Everything consists in proving that such pathological complex structure may be chosen to be a unitary map. Let $\Psi$ be a quasilinear map on $\ell_2$ and let $[x_i]$ be a finite sequence of $n$ normalized vectors. Following \cite{CCFM} we set
$$ \nabla_{[x_i]} \Psi   ={\rm Ave}_{\pm} \left \|  \Psi  \left ( \sn  \pm x_i\right ) - \sn \pm \Psi(x_i)\right \|,
$$
where the average is  taken over all  the signs $\pm 1$. Assume that ${\rm Isom}(L_2)$ is compatible with $\KP$ and let $g \mapsto d(g)$ be the associated derivation. The linearity of $d(g)$ implies that,
$$\nabla_{[x_i]} g\KP \leq \nabla_{[gx_i]}\KP+\nabla_{[x_i]} D(g)$$
where $D(g) = [g, \KP, g] + d(g)$ as in Lemma \ref{info}. The quantity $\nabla_{[x_i]} B(g)$ is bounded by $C\sqrt{n}$ by the paralelogram law. It is proved in \cite[Subsection 3.2 page 9]{CCFM} that there exist two orthonormal sequences of $n$ vectors $[x_i], [y_i]$ such that $\nabla_{[x_i]}\KP=\frac{1}{2}\sqrt{n}\log n$ and
$\nabla_{[y_i]} \KP \leq M\sqrt{n}$ for some uniform constant $M$. Let $g$ be some unitary operator such that $g(x_i)= y_i$, $i=1,\ldots,n$, we get a contradiction for large $n$. The result translates to any infinite dimensional $L_2$ through the fact that the restriction of $\KP$ to an $\ell_2$-subspace generated by disjoint characteristic functions of intervals coincides, up to a linear term, with the own $\KP$ map on $\ell_2$ \cite{jesusjesusfelix}.\end{proof}

\section{$G$-equivariant maps}\label{equi}

The purpose of this section is showing that if $G$-centralizers are connected with interpolation scales of $G$-spaces,
$G$-equivariant maps are connected with \emph{rigid} interpolation scales. Let us give a precise meaning to that word:

\adef An interpolation pair $(X_0, X_1)$ will be called $\theta$-rigid if whenever $Y_0,Y_1 \subset X_0+X_1$ defines another regular pair of interpolation such that $X_\theta=Y_\theta$ isometrically and $\Omega_\theta^X=\Omega_\theta^Y$, it follows that  $X_t=Y_t$ isometrically, for all $0<t<1$. The pair is said to be rigid, if it is $\theta$-rigid for all $0<\theta<1$. \zdef

As it was announced, typical examples of rigid scales are provided by $p$-convexifications of r.i. K\"othe spaces:

\begin{lemma} When $X$ is an r.i. K\"othe space the pair $(X, L_\infty) $ is rigid.\end{lemma}
\begin{proof} In the case of discrete spaces we apply the previous proposition to the open set $U=\{x=(x_i)_i \in \C^n : x_i \neq 0\ \forall i=1,\ldots,n\}$. It is clear that $x \mapsto x \log(|x|)$ is $C_1$ on some neighborhood of any $y \in U$, so the local Lipschitz property will be satisfied. The same idea applies to the case of r.i. spaces on $[0,1]$.\end{proof}

A rigid pair is such that $X_i=Y_i$ isometrically, $i=0,1$, as soon as $\|x\|_i=\lim_{t \rightarrow i} \|x\|_t, i=0,1$ for $x \in X_{0} \cap X_1$, a condition satisfied for most examples (see \cite{M}). It is an open question of \cite{CCFG} whether optimal pairs of interpolation are rigid, even in the special case in which $\Omega_\theta^X$ is bounded. A positive answer was presented in \cite{CCFG} under the assumption $\Omega_\theta^X=0$, or even when $\Omega_\theta^X$ is linear (under technical restrictions). We present a few additional partial answers:

\begin{prop} Assume $(X_0,X_1)$ is an optimal interpolation pair such that either
\begin{enumerate}
\item[(a)] $X_0$ and $X_1$ have a common monotone basis $(e_n)$. In this case we set $E_n=[e_1,\ldots,e_n]$; or
\item[(b)] $X_0$ and $X_1$ are r.i. K\"othe spaces on $[0,1]$. In this case we let $E_n$ be the subspace generated by the
characteristic functions of the intervals $\big[(k-1)/2^n,k/2^n\big]$, $k=1,\ldots,2^n.$
\end{enumerate} Assume the restriction of $\Omega_\theta$ to $S_{X_\theta} \cap E_n$ is locally Lipschitz on a dense open subset  for each $n$. Then the pair $(X_0,X_1)$ is rigid.
\end{prop}

\begin{proof} Pick norm one $x \in X_0 \cap X_1$. By \cite[Proposition 5.3.]{CCFG}, $\Omega_\theta(E_n) \subset E_n$ for each $n$. According to \cite[Theorem 5.11]{CCFG}, the optimal analytic function $B_\theta(x)$ satisfies the differential equation
$F'(t)=i\Omega_\theta(B_\theta(x)(\theta+it))$ with initial  condition $F(0)=x$.
Moreover, $B_\theta(x)$ takes values in $S_{X_\theta}$.\medskip

\noindent \textbf{Claim.} \emph{The equation has a unique holomorphic solution with values in $S_{X_\theta}$ in each of the  cases (a) and (b) for $x$ in the corresponding dense open subset}.\

\begin{proof}[Proof of the Claim] Since $\Omega_\theta$ is locally Lipschitz, if $F$ and $G$ satisfy the differential equation
for $x$ in the dense open subset of $S_{X_\theta} \cap E_n$, then
$$\|F(t)-G(t)\|= \|i\int_0^t \Omega_\theta(F(s))-\Omega_\theta(G(s))\|ds \leq K \int_0^t \|F(s)-G(s)\|$$
for some $K$ and $t$ close enough to $0$. So $\max_{0 \leq s \leq t} \|F(s)-G(s)\| \leq Kt \max_{0 \leq s \leq t}\|F(s)-G(s)\|$
and thus $F(s)=G(s)$ on some small enough interval $[0,t]$. By holomorphy $F=G$. \end{proof}

This means that if we have another regular pair $Y_0,Y_1 \subset X_0+X_1$ such that $X_\theta=Y_\theta$ isometrically and $\Omega_\theta^X=\Omega_\theta^Y$ then the opetimal selectors $B_\theta(x)^X = B_\theta(x)^Y$ coincide and therefore
$$\|x\|^X_t = \|B_\theta(x)^X(t)\|^X_t =  \|B_\theta(x)^Y(t)\|^Y_t = \|x\|^Y_t\qedhere$$\end{proof}

Theorem \ref{Geq-centralizer} admits a version for rigid pairs:

\begin{theorem}\label{Geq-equivariant} Let $(X_0,X_1)$ be a rigid interpolation pair, and let $G$ be a group of isometries on $X_\theta$. Then the following are equivalent:
 \begin{itemize}
\item[(a)] $\Omega_\theta$ defined on $X_\theta$ is $G$-equivariant.
\item[(b)] $G$ acts as an isometry group on the interior of the scale.
\end{itemize} \end{theorem}

\begin{proof} $(b) \Rightarrow (a)$  is Proposition \ref{Gequivariant}. The prof of $(a) \Rightarrow (b)$ goes as that of Theorem \ref{Geq-centralizer} until getting $\mhor_\theta (x) = \Omega_\theta(x)$, where the rigidity hypothesis applies to conclude that  $\|gx\|_t=\|x\|^g_t=\|x\|_t$ for $0<t<1$ and all $g \in G$. \end{proof}

As an example the Kalton-Peck map $\KP$ when defined on a $p$-convex K\"othe space is ${\mathcal U}$-equivariant even if it is not equivariant in the associated $L_\infty$-structure. In general, equivariant quasi-linear maps with respect the the module structure seem to be only possible in trivial cases. It is different for linear maps: an $\mathcal U$- linear map $L: Y\to X$ on spaces having unconditional bases is obviously diagonal since $ge_n=\pm e_n$ are the only options; if the bases are symmetric and $G$ is the group of operators acting by change of signs and permutations of the vectors of a symmetric basis, $G$-linear maps are homotheties. However, combining Corollary \ref{222} and Proposition \ref{Gcentralizer} one gets:
\begin{corollary}\label{hhh} Let $(X_0,X_1)$ be an interpolation pair. Assume $X_\theta$ is reflexive and that $G$ is an amenable group acting on the scale. Then\begin{enumerate}
\item[(a)] $\Omega_\theta$ is boundedly equivalent to a $G$-equivariant map.
\item[(b)] If $\Omega_\theta$ is trivial then it is boundedly equivalent to a $G$-linear map.
\end{enumerate}
\end{corollary}

\section{The category of G-Banach spaces and its exact sequences}\label{exact}

We shift now our point of view from ``compatibility of group actions on twisted sums" to ``equivalence of exact sequences of $G$-spaces". We thus introduce the category \textbf{GBan} of $G$-Banach spaces whose objects are Banach $G$-spaces, and whose arrows
are $G$-operators. An exact sequence in \textbf{GBan} is an exact sequence formed by $G$-Banach spaces and $G$-operators.

An exact sequence of $G$-Banach spaces can be described by a pair $(\Omega, d)$, where $\Omega: Y\lop X$ is quasi-linear and $d$ is an associated derivation that determines the bounded action $\lambda(g)=\left(
                                                                                                       \begin{array}{cc}
                                                                                                         u(g) & d(g) \\
                                                                                                         0 & v(g) \\
                                                                                                       \end{array}
                                                                                                     \right)$ on the twisted sum space
$X\oplus_\Omega Y$. Let us transplant Lemma \ref{info} to this language: The following are objects of \textbf{GBan}: \begin{itemize}
\item $(L, -[u,L,v])$ when $L$ is linear.
\item $(B, T)$ when $B$ and $T$ are bounded.
\item $(\Omega, 0)$ if and only if $\Omega$ is a $G$-centralizer.\end{itemize}

In order to consider maps $\Omega$ defined on a fixed dense $G$-subspace $\Delta\subset Y$ (in particular, $\Delta$ must be $G$-invariant), the role of this $\Delta$ must be examined since the exact sequence of $G$-spaces does not depend on $\Delta$ while the representation $(\Omega, d)$ does. 
We can assume that all the maps involved have a common ambient space $\Sigma$ by the observations we made in `The ambient issue' section. 
Observe the following definitions:

\adef \label{Gequivalence} $\;$
\begin{description}
\item[Equivalence of maps] Let $(\Omega_1, d_1)$ and $(\Omega_2, d_2)$ with $\Omega_1, \Omega_2: \Delta \lop X$ quasi-linear and $d_1,d_2$ their associated derivations. They are $G$-equivalent, something we write $(\Omega_1, d_1) \simeq (\Omega_2, d_2)$ if there is a linear map $L: \Delta \lop X$ such that $\Omega_1 - \Omega_2 - L$ is bounded and $d_1 - d_2 = -[u, L, v]$.
\item[Equivalence of sequences] The sequences generated by $\Omega_1$ and $\Omega_2$ are said to be $G$-equivalent if there is a $G$-operator $T$ making the following diagram commute
$$\xymatrixrowsep{1pc}\xymatrix{&&X\oplus_{\Omega_1} Y\ar[dr]\ar[dd]^T&\\
0\ar[r]&X \ar[ur]\ar[dr]&&Y\ar[r]&0\\
&&X\oplus_{\Omega_2} Y \ar[ur]&&}$$
\end{description}
\zdef

Let us check that the two definitions are equivalent. First observe that if $\Delta$ is dense in $Y$ then $X\oplus_\Omega \Delta$ is dense in $X\oplus_\Omega Y$: if $(\omega, y)\in X\oplus_\Omega Y$ and $\delta\in \Delta$ is such that $\|y-\delta\| = \varepsilon$ then $(\omega + \Omega(\delta -y), \delta)\in X\oplus_\Omega Y$ because $\omega + \Omega(\delta -y) = \omega -\Omega y + \Omega(\delta -y) - \Omega \delta + \Omega y \in X$ and $\|(\omega, y) - (\omega + \Omega(\delta -y), \delta)\| =\|(\Omega(y - \delta), y -\delta)\|=2\varepsilon$. The operator
$\tau=\begin{pmatrix} Id & -L \\ 0 & Id \end{pmatrix}$ makes the diagram
$$\xymatrixrowsep{1pc}\xymatrix{&&X\oplus_{\Omega_1} \Delta\ar[dr]\ar[dd]^\tau&\\
0\ar[r]&X \ar[ur]\ar[dr]&&\Delta\ar[r]&0\\
&&X\oplus_{\Omega_2} \Delta \ar[ur]&&}$$
commute and is a $G$-operator since $$ \begin{pmatrix} u(g) & d_2(g) \\ 0 & v(g) \end{pmatrix}
\begin{pmatrix} Id & -L \\ 0 & Id \end{pmatrix}
=
\begin{pmatrix} Id & -L \\ 0 & Id \end{pmatrix}
\begin{pmatrix} u(g) & d_1(g) \\ 0 & v(g) \end{pmatrix}$$
because  $d_2 - d_1 = [u, L, v]$. Finally, $\tau$ can be extended to a $G$-operator $T: X\oplus_{\Omega_1} Y \To X\oplus_{\Omega_2} Y$ by density: set   $(\omega, y) = \lim (\omega_n, \delta_n)$ and define $T(\omega, y)= \lim \tau(\omega_n, \delta_n)$. Since both actions are continuous,
$\lambda_2 T(\omega, y)= \lambda_2 \lim \tau (\omega_n,\delta_n) = \lim \lambda_2 \tau (\omega_n,\delta_n) = \lim \tau \lambda_1 (\omega_n,\delta_n)= T\lambda_1 \lim (\omega_n,\delta_n) = T \lambda_1 (\omega, y)$. The other implication is immediate: the existence of $T$ already implies equivalence of the exact sequences in the category of Banach spaces, so that $\Omega_1-\Omega_2$ is boundedly equivalent to $L$. Thus, there is a vector space structure on the set of pairs $(\Omega, d)$ (defined on the same $\Delta)$ given by $(\Omega_1, d_1) + (\Omega_2, d_2) = (\Omega_1+ \Omega_2, d_1 + d_2)$ and $\lambda(\Omega, d) = (\lambda \Omega, \lambda d_1)$. The zero element is the class of trivial sequences:

\adef \label{Gtrivial} We will say that $(\Omega, d)$ is $G$-trivial, or that it $G$-splits, if $(\Omega, d)\simeq (0,0)$. This occurs if and only if there is a linear map $L$ such that $\Omega - L$ is bounded and $d = -[u,L,v]$.
\zdef

\begin{prop}\label{fishy} Every quasilinear map $\Omega: \Delta\lop X$ defining a $G$-sequence $0\to X \to X\oplus_\Omega Y \to Y \to 0$ of $G$-spaces
is a linear perturbation of a $G$-centralizer.
\end{prop}
\begin{proof} Let $0\to X\to Z\to Y\to 0$ be an exact sequence of $G$-spaces. Set $\Sigma = Z$ as the ambient space equipped with $\lambda(g)$ as the extension of $u(g)$. Any homogeneous bounded selection $B$ for the quotient map: $B: Y\to Z$ is a $G$-centralizer generating the same sequence; in particular, writing $Z=X\oplus_\Omega Y$ and
$\mhor y= (\Omega y, y)$, its associated bounded action has derivation $0$ since the associated diagonal maps
$\lambda'(g)$ are uniformly bounded on $X \oplus_{\mhor} Y$:
\begin{eqnarray*}
\|\lambda'(g)((x,0), y)\|_{\mhor} &=& \|((u(g)x, 0), v(g)y)\|_{\mhor}\\
 &=& \|( (u(g)x, 0) - ( \Omega v(g)y, v(g)y), v(g)y)\|_{\mhor}\\
 &=& \|(u(g)x - \Omega v(g)y, -v(g)y) \|_\Omega + \|v(g)y\|_Y\\
 &=& \|(u(g)x\|_X + \|v(g)y\|_Y\\
 &\leq& \max(\gamma(u),\gamma(v))(\|x\|_X+\|y\|_Y)
 \\
 &=& \max(\gamma(u),\gamma(v)) \|((x,0), y)\|_{\mhor}. \end{eqnarray*}

We perform the standard pushout from Lemma \ref{perturlin} to get $\xi\mhor = \sigma\Omega +\mathcal L$. Just before that lemma it was observed that $T := \left(
                                                                    \begin{array}{cc}
                                                                      Id & - \mathcal L \\
                                                                      0 & Id \\
                                                                    \end{array}
                                                                  \right)$  makes the following diagram commute:
$$\xymatrixrowsep{1pc}\xymatrix{&&\mathcal X\oplus_{\sigma\Omega} Y \ar[dr]\ar[dd]^T
&\\
0\ar[r]&\mathcal X \ar[ur]\ar[dr]&&Y\ar[r]&0\\
&& \mathcal X\oplus_{\xi\mhor} Y\ar[ur]&&}.$$
It only remains to see that $T$ is a $G$-operator, i.e.
\begin{eqnarray*}
(u(g)x-u(g)\mathcal Ly, v(g) y)&=&  \left(\begin{array}{cc}
                                                                      u & 0 \\
                                                                      0 & v \\
                                                                    \end{array}
                                                                  \right) \left(\begin{array}{cc}
                                                                      Id & -\mathcal L \\
                                                                      0 & Id \\
                                                                    \end{array}
                                                                  \right) \left(
                                                                            \begin{array}{c}
                                                                              x \\
                                                                              y \\
                                                                            \end{array}
                                                                          \right)\\
                                                                  &=& \left(\begin{array}{cc}
                                                                      Id & -\mathcal L \\
                                                                      0 & Id \\
                                                                    \end{array}
                                                                  \right)\left(\begin{array}{cc}
                                                                      u & d \\
                                                                      0 & v \\
                                                                    \end{array}
                                                                  \right)\left(
                                                                            \begin{array}{c}
                                                                              x \\
                                                                              y \\
                                                                            \end{array}
                                                                          \right)\\
                                                                          &=&(u(g)x+ d(g)y -\mathcal Lv(g)y, v(g) y)\end{eqnarray*}
namely, $u(g)\mathcal Ly = -d(g)y +\mathcal Lv(g)y$ or, what is the same, $d= -[u,\mathcal L,v]$, which is immediate from $\sigma\Omega = \xi\mhor - \mathcal L$, and the fact that $\Omega'$ is a $G$-centralizer.\end{proof}


We now give two easy lemmas that will help us simplify some proofs later on.

\begin{lemma} \label{trivialG} Let $0\to X \to X\oplus_\Omega Y \to Y \to 0$ be a trivial exact sequence $(\Omega, d)$ of $G$-spaces. If $L: \Delta\lop X$ is any linear map for which $\Omega - L: \Delta\to X$ is bounded, then $d(g) + [u(g), L, v(g)]$ is a uniformly bounded family of operators.\end{lemma}
\begin{proof}
$d(g)+[u(g),L,v(g)]=
(d(g)+[u(g),\Omega,v(g)]) +
[u(g),L-\Omega, v(g)]$
and both terms of the sum are uniformly bounded. \qedhere
\end{proof}


\begin{lemma}\label{bounded0} If $B:Y\to X$ is a bounded map then $(B, d)\simeq (0,d)$\end{lemma}
\begin{proof} It is clear that the formal identity map $X\oplus_B Y \to X\oplus Y$ is a $G$-operator.\end{proof}

A warning is perhaps judicious here: sometimes, quasilinear maps $\Omega: Y\lop X$ are bounded maps $Y\to \Sigma$ but that does not imply that $\Omega$ is equivalent to $0$, let alone $(\Omega, d)\simeq (0,d)$: beware that if $\Omega$ is not bounded {\em with respect  to the $\|.\|_X$-norm}, no identity map $X\oplus_\Omega Y\to X\oplus Y$ exists. $G$-splitting admits natural characterizations, similar to those in the Banach space category. Let us say that a $G$-subspace of a $G$-space is $G$-complemented when it is complemented by a $G$-projection.

\begin{prop}\label{Gsplits} Consider an exact sequence $(\Omega, d)$ of $G$-spaces $0 \rightarrow X \rightarrow Z \rightarrow Y \rightarrow 0$. The following assertions are equivalent:
\begin{itemize}
\item[(i)] The sequence $G$-splits.
\item[(ii)] The quotient map admits a linear continuous $G$-lifting.
\item[(iii)] $X$ admits a $G$-invariant complement.
\item[(iv)] $X$ is $G$-complemented in $Z$.
\end{itemize}
\end{prop}

\begin{proof} A few hints will suffice: If $L$ is a $G$-lifting then $L[Y]$ is a $G$-complement of $X$; $L(y)=(\ell y, y)$ with $\ell-\Omega$ bounded and  $d = -[u, \ell, v]$ is a derivation.\end{proof}

In complete analogy with classical Banach space homology, we can define now the vector space $\Ext_G(Y,X)$ of $G$-equivalence classes of  pairs $(\Omega, d)$. Our next result presents an optimal ``group" version of two theorems of Cabello \cite[Cor. 2]{felix} and \cite{felix2}: the first one asserts that an exact sequence of $L_\infty$-modules that algebraically splits also splits topologically; the second says that when $p\neq q$ the only exact sequence of quasi-Banach $L_\infty$-modules $0\to L_q\to Z\to L_p\to 0$ is the trivial one while, while, as it is well known \cite{cabecastuni}, $\Ext(L_p, L_q)\neq 0$ as Banach spaces. In striking contrast, we prove:

\begin{theorem}\label{Gsame} Let $G$ be a group and let $0\to X \to X\oplus_\Omega Y \to Y \to 0$ be a trivial exact sequence $(\Omega, d)$ of $G$-spaces. If $G$ is amenable and $X$ is a $G$-ultrasummand then the sequence $G$-splits.\end{theorem}

\begin{proof} Assume that $\Omega$ is trivial. We use Proposition \ref{baba} to obtain a $G$-centralizer $\mhor$ so that $(\Omega, d) \simeq (\mhor, 0)$. Since we are told that $(\mhor, 0)$ splits, Lemma \ref{trivialG} yields a linear map $\tau:\Delta \to \Sigma$ for which $\mhor - \tau: \Delta \to X$ is bounded and $(\mhor, 0)\simeq (\mhor - \tau, [u,\tau, v]) \simeq (0, [u, \tau, v])$ by Lemma \ref{bounded0}.
Thus, the proof can be reduced to proving that, under the hypothesis of the theorem, if $B:\Delta \to X$ is a bounded map then
$(B, d)\simeq (B, 0)$ for whatever $d$. Recall from Lemma  \ref{trivialG} that $\{d(g)\}_{g\in G}$ is a uniformly bounded family of operators, and therefore we can form the operator
$$My  =P\left (\int_{g \in G}  u(g^{-1})d(g)y\; d\mu(g)\right)$$
where $P: X^{**} \rightarrow X$ is a $G$-projection.
Let us show that the map $R=\begin{pmatrix} Id & M\\ 0 & Id \end{pmatrix}$ is
a $G$-operator making the diagram
$$\xymatrixrowsep{1pc}\xymatrix{&&X\oplus_{B} Y\ar[dr]\ar[dd]^R&\\
0\ar[r]&X \ar[ur]\ar[dr]&&Y\ar[r]&0\\
&&X\oplus_{B} Y\ar[ur]&&}$$
commute. The only part that is not evident, that $R$ is a $G$-operator means $$ \left(
\begin{array}{cc}
                                                                                                                                    u & 0 \\
                                                                                                                                     0 & v \\
                                                                                                                                   \end{array}
                                                                                                                                 \right)R = R  \left(
                                                                                                                                   \begin{array}{cc}
                                                                                                                                     u & d \\
                                                                                                                                     0 & v\\
                                                                                                                                   \end{array}
                                                                                                                                 \right)$$
namely $u(g)M = d(g)+ Mv(g)$, i.e., $d=[u,M,v]$. We show this:

\begin{eqnarray*} u(g')My &=& u(g')P\left (\int_{g \in G}  u(g^{-1})d(g)y\; d\mu(g)\right)\\
&=&P\left (\int_{g \in G}  u(g'g^{-1})d(g)y \; d\mu(g)\right)\\
\end{eqnarray*}
Call $g'g^{-1} = h^{-1}$ so that $g=hg'$ and thus
\begin{eqnarray*}
&=& P\left (\int_{h \in G}  u(h^{-1})d(hg') y\; d\mu(h)\right)\\
&=& P\left (\int_{h \in G}  u(h^{-1})(u(h)d(g') + d(h)v(g')) y\; d\mu(h)\right)\\
&=& d(g') + P\left (\int_{h \in G}u(h^{-1})d(h)v(g')) y\; d\mu(h)\right)\\
&=& d(g') + Mv(g').\qedhere \end{eqnarray*}\end{proof}

\begin{corollary} Let $G$ be a group.  Let $0\to X \to X\oplus_\Omega
({\rm resp.\ }\oplus_\Phi) Y \to Y \to 0$ be an exact sequence $(\Omega, d)$ (resp. $(\phi, d')$) of $G$-spaces. If $G$ is amenable and $X$ is a $G$-ultrasummand then $(\Omega, d) \simeq (\Phi, d') \Longleftrightarrow \Omega \sim \Phi$.\end{corollary}

Both hypothesis are necessary:
(a) The amenability of $G$ is necessary since $\Ext_{{\rm Aut}(\mathbb T)}(\ell_2(\mathbb T), \ell_2(\mathbb T)$ $\neq 0$: indeed, the sequence $0\to \ell_2(\mathbb T)\to \ell_2(\mathbb T)\to \ell_2(\mathbb T)\to 0$ does not ${\rm Aut}(\mathbb T)$-splits since it has the form $(L,0) $ and, otherwise, $(0, -[u , L])$ would be trivial and then $[u, L]$ would be an inner derivation, a contradiction with the fact that the corresponding representation of $\F_\infty$ from \cite{PS} is non-unitarizable.
(b) The $G$-ultrasummand character of $X$ is necessary since we will show in Section \ref{cantor} that $\Ext_{2^{<\omega}}(\R, c_0)\neq 0$.

Theorem \ref{Gsame} seen together with the result of Cabello mentioned before its statement may seem surprising, since the group $\mathcal U$ of units of the $L_\infty$-module structure is abelian and (for $1<p<\infty$) $L_p$ spaces are reflexive. Let us spell  what these two results together actuallly imply: no non-trivial element of ${\rm Ext}(L_p,L_q)$ can be compatible with the canonical actions of $\mathcal U$ on these two spaces.

A significant consequence of Theorem \ref{Gsame} is a kind of uniqueness result for the possible derivation associated to fixed actions. It will help us at this point to use the classical terminology, which calls {\em inner}  a derivation for which there exists a bounded linear map $A: Y \rightarrow X$ such that $d = [u,A,v]$. Thus, $(\Omega, d_1) \simeq (\Omega, d_2)$ if and only if $d_1 -d_2$ is inner. In the particular case of a direct sum of two copies of a Hilbert space $0 \rightarrow H_1 \rightarrow H_1 \oplus H_2 \rightarrow H_2 \rightarrow 0,$  we recover in this way the classical fact that the representation $\lambda$ is similar to the diagonal unitary representation given by $u,v$ if and only if $d$ is inner. Proposition \ref{Gsplits} may be seen as a generalization of this fact for triangular representations on twisted sums:

\begin{corollary}\label{uniqe} Assume that $G$ is an amenable group, $Y,X$ are $G$-spaces with $X$ a $G$-ultrasummand and $\Omega: Y\lop X$ is a quasilinear map. All compatible actions of $G$ on $X \oplus_\Omega Y$ are conjugate; namely, given two actions $\lambda_1, \lambda_2$ there is $A \in \mathfrak L(Y,X)$ such that for all $g \in G$,
$$\lambda_2(g)=\begin{pmatrix} Id & A \\ 0 & Id \end{pmatrix}
\lambda_1(g)
\begin{pmatrix} Id & -A \\ 0 & Id \end{pmatrix}.$$
\end{corollary}
\begin{proof} Since $d_1-d_2$ is inner, pick $A$ such that
 $d_1(g)-d_2(g)=[u(g),A,v(g)]$.
\end{proof}

\section{Variations and comments}
This final section contains a miscellanea of results and problems connected with the ideas in this paper.

\subsection{From uniformly bounded extensions to actions}
The following situation was mentioned in the abstract: to which extent the existence of a uniformly bounded family of operators
on a twisted sum space compatible with a couple of actions on the subspace and the quotient space induces an action on the twisted sum. We have:

\begin{prop} Let $\Omega: Y\lop X$ be quasi-linear between two $G$-spaces. Assume that there is a uniformly bounded
family of operators $(T_g)_{g\in G}$ such that each $T_g: X \oplus_\Omega Y\to X \oplus_\Omega Y$ forms a commutative diagram
$$\begin{CD}
0@>>> X @>>> X\oplus_\Omega Y @>>> Y @>>> 0\\
&&@V{u(g)}VV @V{T_g}VV @VV{v(g)}V\\
0@>>> X @>>> X\oplus_\Omega Y @>>> Y @>>> 0\end{CD}$$
If $G$ is an amenable group and $X$ is a $G$-ultrasummand then there is a compatible action of
$G$ on $X\oplus_\Omega Y$.\end{prop}

\begin{proof} Each operator $T_g$ has the form $T_g = \left(
                                                        \begin{array}{cc}
                                                          u(g) & \tau_g \\
                                                          0 & v(g) \\
                                                        \end{array}
                                                      \right).$ We may assume wlog that $\tau_{e}=0$ and from that $T_{g^{-1}} = \begin{pmatrix} u(g) & \tau_g \\ 0 & v(g) \end{pmatrix}^{-1}=
\begin{pmatrix} u(g^{-1}) & -u(g^{-1}) \tau_g v(g^{-1}) \\ 0 & v(g^{-1}) \end{pmatrix} = T_g^{-1}$, namely
$$\tau_{g^{-1}} = -u(g^{-1}) \tau_g v(g^{-1}).$$
It is easy to check that the family $\{ [u(g),\Omega, v(g)] + \tau_g\}_{g\in G}$ is uniformly bounded, which implies that the family
$\{u(g)\tau_h + \tau_g v(h) - \tau_{gh}\}_{h\in G}$ is uniformly bounded, and thus after we check that also
$\{u(gh)\tau_{h^{-1}} + \tau_{gh}v(h^{-1}) + \left( u(g)\tau_h + \tau_g v(h) - \tau_{gh}\right) v(h^{-1})\}_{h\in G}$ is uniformly bounded, it follows that
 $\{u(gh)\tau_{h^{-1}} + \tau_{gh}v(h^{-1})\}_{h\in G}$ s uniformly bounded:
\begin{eqnarray*}
\tau_g&=&u(g)\left( u(h)\tau_{h^{-1}} + \tau_h v(h^{-1})\right) + \tau_g\\
&=&u(gh)\tau_{h^{-1}}  + \left( u(g)\tau_h + \tau_g v(h)\right) v(h^{-1})\\
&=&u(gh)\tau_{h^{-1}} + \tau_{gh}v(h^{-1}) + \left( u(g)\tau_h + \tau_g v(h) - \tau_{gh}\right) v(h^{-1}).
\end{eqnarray*}
We can therefore define
$$d(g)=P\left (\int_{h \in G} u(gh)\tau_{h^{-1}}+ \tau_{gh}v(h^{-1}) d\mu\right),$$
where $P$ is a $G$-operator $X^{**} \rightarrow X$,
and compute that $d$ is a derivation: given   $g,k$ in $G$, since $\tau_h v(h^{-1}k)+u(h)\tau_{h^{-1}}v(k) =(\tau_h v(h^{-1})-u(h)u(h^{-1}) \tau_h v(h^{-1}))v(k)=0$ one has:
\begin{eqnarray*}
d(gk)&=& P\left (\int_{h \in G} u(gkh)\tau_{h^{-1}}+ \tau_{gkh}v(h^{-1}) d\mu\right)\\
&=& P\left(\int_{h \in G} u(gkh) \tau_{h^{-1}} + \tau_{gkh} v(h^{-1}) +
 u(g) \tau_{h} v(h^{-1}k) + u(g)u(h)\tau_{h^{-1}} v(k)\right)\\
&=&P(\int_{h \in G} u(gkh)\tau_{h^{-1}} +u(g)\tau_{kh}v(h^{-1})+u(gh)\tau_{h^{-1}}v(k)+\tau_{gh} v(h^{-1}k))\\
&=&u(g)P\left(\int_{h \in G} u(kh)\tau_{h^{-1}} +\tau_{kh}v(h^{-1}) d\mu \right) +
P\left(\int_{h \in G} u(gh)\tau_{h^{-1}}+\tau_{gh} v(h^{-1}) d\mu\right) v(k)\\
&=& u(g)d(k)+d(g)v(k)\qedhere\end{eqnarray*}
\end{proof}

It seems to be open whether the amenability and $G$-ultrasummand hypotheses are necessary.

\subsection{Complex structures.} We now answer now a question complex structures (i.e. operators of square $-{Id}$) on real twisted sum spaces posed in \cite{CCFM}.

\begin{prop} Let $X,Y$ be Banach spaces admitting complex structures $u,v$ and let $\Omega: Y\lop X$ be a quasilinear map. If there exists a bounded operator $T$ on $X \oplus_\Omega Y$ yielding a commutative diagram
$$\begin{CD}
0@>>> X @>>> X\oplus_\Omega Y @>>> Y @>>> 0\\
&&@V{u}VV @V{T}VV @VV{v}V\\
0@>>> X @>>> X\oplus_\Omega Y @>>> Y @>>> 0\end{CD}$$
then $X\oplus_\Omega Y$ admits a complex structure.\end{prop}

\begin{proof} We use the commutative, hence amenable, group $G=\{i,-1,-i,1\}$ for which no $G$-complementation is required since one performs just a finite average, providing that if $\begin{pmatrix} u & \tau \\ 0 & v \end{pmatrix}$ is bounded then
$\begin{pmatrix} u & \frac{1}{2}(\tau+ u \tau v) \\ 0 & v \end{pmatrix}$ is a compatible complex structure.\end{proof}

This proof shows that complex structures exist in $X\oplus_\Omega Y$ as long as $[u, \Omega, v]$ is the sum of a bounded plus a linear map.
The result had been proved in \cite[Corollary 2.2]{CCFM} assuming that $[u,\Omega, v]$ was either bounded or linear.




\subsection{Why singular centralizers on $L_p$ do not exist.} Singular quasilinear maps (those whose restriction to infinite dimensional subspaces is never trivial) are somewhat mysterious objects with many potential applications, whose paramount example is the Kalton-Peck map on $\ell_p$ spaces (but not the Kalton-Peck map on $L_p$ spaces).  The key result \cite{cabecen} is that no singular $L_\infty$-centralizer exists on $L_p$, a result generalized in \cite{CCFM2} to superreflexive K\"othe space over a non-atomic base and the proof consists in showing that there is a copy of $\ell_2$ contained in the domain of the centralizer. In the case of $L_p(0,1)$, it is not hard to see any Rademacher function is in the domain of $\KP$ and then use almost-transitivity plus Proposition \ref{translemma} to get that $\KP$ is not singular on $L_p$.

\subsection{Actions of the Cantor group $2^\omega$ and of $2^{<\omega}$}\label{cantor} The Cantor group is the group of units of $\ell_\infty$ and thus $2^\omega$-centralizers are just $\ell_\infty$-centralizers. Its diagonal action on $\ell_\infty$ restricted to $c_0$ is again the diagonal action, and thus it generates an action on $\ell_\infty/c_0$. We do not have any reasonable idea about a linear derivation $d: \ell_\infty/c_0\to c_0$ of the Cantor group. The subgroup $2^{<\omega}$ of elements of $2^\omega$ that are eventually $1$ is much more manageable. The space $c$ is the living example that $2^{<\omega}$-groups are not $2^{\omega}$-groups. The natural diagonal action of $2^{<\omega}$ on $c$ and $c_0$, that is therefore a $2^{<\omega}$-subspace, induces the identity action on the quotient $\R$. This implies that the exact sequence $0\to c_0 \to c \to \R \to 0$ of $2^{<\omega}$-spaces, which splits as a Banach space sequence, does not split as a $2^{<\omega}$-sequence since no $2^{<\omega}$-lifting $\R\to c$ is possible, and thus $\Ext_{2^{<\omega}}(\R, c_0)\neq 0$, which shows that $G$-complementation is necessary in Theorem \ref{Gsame}. Since the triangular action on $c$ has the form $\lambda(g)= \begin{pmatrix} u(g) & d(g) \\ 0 & Id_\R \end{pmatrix} $ with $u$ the diagonal action and $d(g): \R \rightarrow c_0$ is $d(g)(r) = r \sum_{g_i = - 1} e_i$, this $d(g)$ is a linear derivation on $2^{<\omega}$. We do not know if all actions of $2^{<\omega}$ on $c$ have the form $\begin{pmatrix} u(g) & xd(g) \\ 0 & Id_\R \end{pmatrix}$ for $x \in c$ since Corollary \ref{uniqe} does not apply here. Thus, all elements of $\mathfrak L(\R, c)=c$
define $2^{<\omega}$-centralizers and solving the equation $[u(g), L, g ] r = d(g)r$ yields that all $L(r) = - \frac{r}{2}x$, for $x\in c$, are equivariant $2^{<\omega}$-centralizers. There is a general formulation for this situation: let $X$ be a separable Banach space that we write as $X=\overline {\bigcup_n F_n}$ for an increasing sequence of finite dimensional spaces $F_n$. The space $c_0(F_n)$ admits a natural ``diagonal" action $g(f_n)= (g(n)f_n)$ that naturally extends to the space $c_X(F_n)= \{(f_n): \exists \lim f_n\}$. What is interesting here is that the exact sequence $0 \to  c_0(F_n) \to c_X(F_n) \stackrel{\lim}\to  X \to 0 $ splits if and only if $X$ has the Bounded Approximation Property although it never $2^{<\omega}$-splits since the action induced on $X$ is the identity. All this was based on some ideas of \cite{AFGR}, where an example of an SOT-discrete bounded group of operators on $c_0$ without discrete orbits was provided; the relation with twisted sums  was not observed there.

The difficulty of obtaining derivations $\ell_\infty/c_0\to c_0$ for $2^\omega$ can be confronted with how easily one obtains derivations
for $2^{<\omega}$ on the subspace (here $\mathfrak c$ is the cardinal of the continuum) $c_0(\mathfrak c)$ of $\ell_\infty/c_0$. Consider to this end the Nakamura-Kakutani (see \cite{hmbst}) sequences $0 \To c_0 \To C(\Delta_{\mathcal A} ) \To  c_0(|\mathcal A|) \To 0$ also provide natural examples of $2^{<\omega}$-centraizers: pick $\mathcal A$ an almost disjoint family of subsets of $\N$ (i.e., $|A\cap B|<\infty$ for all $A, B \in \mathcal A$) containing the singletons. The cardinal of the family must be $\aleph_1 \leq |\mathcal A|\leq  \mathfrak c$ since when $|\mathcal A|=\aleph_0$ the sequence splits by Sobczyk's theorem.
We will assume without loss of generality that it is the continuum. Let $\Delta_{\mathcal A} $ be the one-point compactification of the locally  compact space having $\N$ as isolated points and $A\in \mathcal A$ as the only accumulation point of $\{n: n\in A\}$. There is a natural action of $2^{<\omega}$ on $C(\Delta_{\mathcal A})$: $(gf)(n)= g(n)f(n)$ that yields the diagonal action on $c_0$ and induces the identity action on $c_{0}(\mathfrak c)$. Let $c_{00}(\mathfrak c)$ be the dense subspace of all finitely supported sequences. A quasilinear map $\Omega:  c_{00}(\mathfrak c)\lop c_0$ corresponding to this sequence can be easily described: fix a well-order on $\mathfrak c$ and then
for $x\in c_{00}(\mathfrak c)$ write it as $x=\sum \lambda_i e_i$ with the $e_i$ well ordered and define $\Omega( \sum \lambda_i e_i) = \lambda_1 1_{A_1} + \lambda_2 1_{A_2 \setminus A_1} + \cdots + \lambda_n 1_{A_n \setminus (A_1\cup \dots A_{n-1})}$. This is a bounded map $c_{00}(\mathfrak c)\to \ell_\infty$ and therefore a $2^{<\omega}$-centralizer (with derivation 0). On the other hand, $C(\Delta_{\mathcal A})$ is a subspace of $\ell_\infty$ but the natural action of $2^\omega$ does not respect $C(\Delta_{\mathcal A})$.

\subsection{Groups and symmetries} To fix ideas, let us focus on $\N$ and $\ell_\infty$-centralizers on Banach spaces with symmetric basis. A centralizer is symmetric if $\|(\Omega x)\sigma - \Omega (x\sigma)\|\leq C \|x\|$ for every permutation $\sigma$ of $\N$. For instance $\KP$ maps are symmetric. Symmetric centralizers live their own lives (see \cite{K,felix2}) and there is a great difference between working with symmetric and non-symmetric centralizers. But the ideas in this paper allow us, once the action of a group $G$ on a
space with unconditional basis has been established, to also consider the action of the group $G_\Theta= G\times \Theta$, where $\Theta$ is a certain group of permutations of $\N$ in the form $(g,\sigma)x= g(x\sigma)$. When $\Theta$ is the whole group of permutations we will call the group $G_{sym}$. Thus, symmetric centralizers are $2^\omega_{sym}$-centralizers. But other groups of permutations are also important: let $(A_n)$ be a partition $\N=\cup A_n$ of $\N$ into finite sets, $A_n<A_{n+1}$, and let $\Theta$ be the group of
permutations $\sigma$ of $\N$ such that $\sigma A_n = A_n$ for all $n$. It turns out that $G_{\Theta}$-symmetric centralizers are sometimes useful: as the authors of \cite{kaltfou} dismayingly recall, the first author has frequently asked about ``how many"
``different" exact sequences $0\to \ell_1\to Z \to c_0\to 0$ exist. The same problem is addressed in \cite{hmbst}. Let us lodge the problem
in the theory developed in this paper.

\begin{prop} Non equivalent $G_{\Theta}$-centralizers $\ell_2\to \ell_2$ generate non-equivalent exact sequences
$0\to \ell_1\to X \to c_0\to 0$.\end{prop}
\begin{proof} By Theorem \ref{Gsame} it is enough to show that a nontrivial $G_{\Theta}$-centralizer $\Omega$ on $\ell_2$ generates a nontrivial exact sequence $0\to \ell_1\to X \to c_0\to 0$. The sequence
$$\begin{CD}0@>>> \ell_2(A_n) @>>> \ell_2(A_n) \oplus_{\Omega|_{A_n}} \ell_2(A_n) @>>> \ell_2(A_n) @>>> 0\end{CD}$$
splits, but if $\sigma_n^{-1}$ denotes its splitting constant, $\lim \sigma_n =0$ since otherwise $\Omega$ would be
$G_{\Theta}$-trivial. Some subsequence $(\sigma_{k(n)})_n\in \ell_1$, and we will shamelessly assume that it is $\sigma$. Let $D: c_0(\N, \ell_2(A_n)) \to \ell_2(\N, \ell_2(A_n))$ be the ``diagonal" map $D((x_n))=(\sigma_n^{1/2} x_n)$. Form the commutative diagram
$$\begin{CD}
0@>>> \ell_2(\N, \ell_2(A_n)) @>>> \ell_2(\N, \ell_2(A_n) \oplus_\Omega \ell_2(A_n) ) @>>> \ell_2(\N, \ell_2(A_n)) @>>> 0\\
&&@VDVV&&@AADA\\
0@>>> \ell_1(\N, \ell_2(A_n)) @>>> X @>>> c_0(\N, \ell_2(A_n)) @>>> 0
\end{CD}$$ and let us call $\ell_2(\Omega)$ the quasilinear map associated to the upper sequence. The lower sequence has $D\ell_2(\Omega) D$ as associated quasilinear map. Since $D\ell_2(\Omega)D$ is $G_{\Theta}$-equivalent to $\Omega$, the sequence does not split. If $\jmath_n: \ell_2(A_n)\to \ell_1^{2^n}$ is a sequence of $C$-isomorphic embeddings
then $(\jmath_n) D \ell_2(\Omega) D (\jmath_n^*)$ is a nontrivial (see \cite{ccky}) sequence
$$\begin{CD}
0@>>> \ell_1 = \ell_1(\N, \ell_1^{2^n}) @>>> X @>>> c_0(\N, \ell_\infty^{2^n}) = c_0@>>> 0\qedhere\end{CD}$$
\end{proof}

Even if $D\ell_2(\Omega) D$ is a $2^\omega$-centralizer, $(\jmath_n) D \ell_2(\Omega)D (\jmath_n^*)$ is not, and can never be, a $2^\omega$-centralizer: otherwise there would be a compatible action of $2^\omega$ on $X$ and picking any extension $T: X\to \R$ of the sum functional $\ell_1\to \R$ we can form the $2^\omega$-invariant functional $\Lambda(x) = \int_{2^\omega} \varepsilon^{-1}T(\varepsilon x) d\mu$. The road is now paved to define $Q: X\to \ell_1^{**}$ in the form $Q(x)(\varepsilon) = \Lambda(\varepsilon x)$ for $\varepsilon$ an unit of $\ell_\infty$ and extend it linearly to a functional on $\ell_\infty$. Finally, compose with a $2^\omega$-projection $\ell_1^{**}\to \ell_1$. It is however perfectly reasonable to have a $G$-centralizer $\Omega$ and two operators $\alpha, \gamma$ so that $\alpha \Omega \gamma$ is a $G'$-centralizer for two different groups $G,G'$. Researchers willing to  travel this road are advised to do so crossing through the horn gate of \cite{kaltfou}.

\subsection{Additional structures} Other additional structures than group structures may be considered on Banach spaces. See for example the work of Corr\^ea \cite{correa} on exact sequences of operator spaces and a solution to $3$-space problem for OH spaces. It seems to be open whether a relevant theory of groups acting completely boundedly on extension sequences of operator spaces may be developed.


\begin{thebibliography}{99}
\bibitem{AFGR} L. Antunes, V. Ferenczi, S. Grivaux and Ch. Rosendal, {\em Light groups of isomorphisms of Banach spaces
and invariant LUR renormings}, Pacific J. Math., 301 (1) (2019), 31--54.

\bibitem{belo} J. Bergh and J. L\"ofstr\"om. \emph{Interpolation spaces. An introduction,} Springer, 1976.

\bibitem{felix} F. Cabello S\'anchez, {\em Nonlinear centralizers with values in $L_0$}, Nonlinear analysis 88 (2013) 42--50.
\bibitem{felix2} F. Cabello S\'anchez, {\em Nonlinear centralizers in homology}, Math. Ann. 358 (2014) 779--798.
\bibitem{cabecen} F. Cabello S\'anchez, \emph{There is no strictly singular centralizer on $L_p$,} Proc. Amer. Math. Soc. 142 (2014) 949--955.
\bibitem{noncom} F. Cabello S\'anchez, {\em The noncommutative Kalton-Peck spaces}, J. Noncomm. Geom. 11 (2017) 1395-1412.
\bibitem{cabecastuni} F. Cabello S\'anchez, J.M.F. Castillo, \emph{Uniform boundedness and twisted sums of Banach spaces}, Houston J. Math. 30 (2004) 523--536.

\bibitem{hmbst} F. Cabello S\'anchez, J.M.F. Castillo, \emph{Homological methods in Banach space theory}, Cambridge Studies in Advanced Mathematics, 2021. Online ISBN. 9781108778312.


\bibitem{cck} F. Cabello S\'anchez, J.M.F. Castillo, N.J. Kalton, \emph{Complex interpolation and twisted twisted Hilbert spaces},
Pacific J. Math. 276 (2015) 287 - 307.
\bibitem{ccky} F. Cabello S\'anchez, J.M.F. Castillo, N.J. Kalton, D. Yost, \emph{Twisted sums of $C(K)$-spaces} Trans. Amer. Math. Soc. 355 (2003) 4523-4541.
\bibitem{jesusjesusfelix} F. Cabello S\'anchez, J.M.F. Castillo, J. Su\'arez, {\em On strictly singular nonlinear centralizers}, Nonlinear Anal. - TMA 75 (2012) 3313--3321.

\bibitem{kaltfou} F. Cabello S\'anchez, A. Salguero, \emph{When Kalton and Peck met Fourier}, arXiv: 2101.11561.
\bibitem{CZ} A.P. Calder\'on, {\em Intermediate spaces and interpolation, the complex method}, Studia Math. 24 (1964) 113--190.
\bibitem{Carro} M. Carro, {\em Commutators and analytic families of operators}, Proc. Roy. Soc. Edinburgh, 129 (4), 1999, 685--696.
\bibitem{CCFG} J.M.F. Castillo, W.H.G. Corr\^ea, V. Ferenczi, M. Gonz\'alez, {\em On the stability of the differential process generated by complex interpolation}, J. Inst.  Math.  Jussieu, (2020) 1-32. doi:10.1017/S1474748020000080.



\bibitem{CCFM} J.M.F. Castillo, W. Cuellar, V. Ferenczi, Y. Moreno, {\em Complex structures on twisted Hilbert spaces}, Israel J. Math., 222 (2017), 787--814.

\bibitem{CCFM2} J.M.F. Castillo, W. Cuellar, V. Ferenczi, Y. Moreno, \emph{On disjointly singular centralizers}, Israel J. Math., (to appear) arXiv:1905.08241.
\bibitem{CFG} J.M.F. Castillo, V. Ferenczi, M. Gonz\'alez, \emph{Singular twisted sums generated by complex interpolation}, Trans. Amer. Math. Soc.  369 (2017), 4671--4708.
\bibitem{correa} H.W.G. Corr\^ea, {\em Twisting operator spaces}, Trans. Amer. Math. Soc., 370 (2018),  8921-8957.
\bibitem{correa2} H.W.G. Corr\^ea, \emph{Complex interpolation of families of Orlicz sequence spaces}, Israel J. Math., to appear. arXiv 1906.01013.
\bibitem{correa3} H.W.G. Corr\^ea, \emph{Complex interpolation of Orlicz sequence spaces and its higher order Rochberg spaces}, arXiv 2105.11238.
\bibitem{cowie} E. Cowie, {\em A note on uniquely maximal Banach spaces}, Proc. Amer. Math. Soc.,  26 (1983), 85--87.

\bibitem{CJMR} M. Cwikel, B. Jawerth, M. Milman, and R. Rochberg, {\em Differential estimates and commutators in interpolation theory}, Analysis at Urbana (Earl R. Berkson, N. Tenney Peck, and J. Jerry
Uhl, eds.), vol. 2, Cambridge University Press, 1989, Cambridge Books Online,  170--220.
\bibitem{D} M. Daher, {\em  Hom\'eomorphismes uniformes entre les sph\`eres unit\'e des espaces d'interpolation}, Canad. Math.
Bull. 38 (1995), 286-294.
\bibitem{day} M. Day, {\em Means for the bounded functions and ergodicity of the bounded representations of semi-groups},
Trans. Amer. Math. Soc. 69 (1950), 276--291.

\bibitem{dix} J. Dixmier, {\em Les moyennes invariantes dans les semi-groupes et leurs applications}, Acta Sci. Math. (Szeged) 12 (1950), 213--227.

\bibitem{EM}  L. Ehrenpreis and F. I. Mautner, {\em Uniformly bounded representations of groups}, Proc. Nat. Acad. Sci. U. S. A 41 (1955), 231--233.
\bibitem{FR2} V. Ferenczi and Ch. Rosendal, {\em Non-unitarisable representations and maximal symmetry}, J. Inst. Math. Jussieu 16 (2017) 421--445.
\bibitem{isometries} R. Fleming and J. Jamison, {\em Isometries on Banach spaces}. Vol. 1. Function spaces, Monographs and Surveys in Pure and App. Math., 129. Chapman and Hall/CRC, Boca Raton, FL, 2003.


\bibitem{kalt} N.J. Kalton, \emph{The three-space problem for locally bounded F-spaces}, Compositio Math. 37 (1978) 243--276.

\bibitem{Ksymplectic} N.J. Kalton, {\em The space $Z_2$ viewed as a symplectic Banach space}, Proceedings of Research Workshop on Banach space theory (1981) Bohr-Luh Lin ed., Univ. of Iowa 1982, 97--111.

\bibitem{kaltnachr} N.J. Kalton,
\emph{Locally complemented subspaces and $\mathcal L_p$ spaces for $p<1$}, Math. Nachr. 115 (1984)
71--97.


\bibitem{kaltonams} N.J. Kalton, {\em Nonlinear commutators in interpolation theory}, Mem. Amer. Math. Soc. 385, 1988.

\bibitem{K} N.J. Kalton, \emph{Differentials of complex interpolation processes for K\"othe function spaces,}
Trans. Amer. Math. Soc. 333 (1992) 479--529.

\bibitem{handbook} N.J. Kalton and S. Montgomery-Smith, {\em Interpolation of Banach spaces}, Handbook of the Geometry of Banach spaces, vol. 2, Elsevier 2003, 1131--1175.

\bibitem{KP} N.J. Kalton and N.T. Peck, {\em Twisted sums of sequence spaces and the three-space problem}, Trans. Amer. Math. Soc. 255 (1979) 1-30.


\bibitem{kuch} P. Kuchment, \emph{Thre-representation problem in Banach spaces}, Complex Analysis and Operator Theory (2021)
doi: 10.1007/s11785-021-01079-6.

\bibitem{M} M. Milman, {\em Notes on limits of Sobolev spaces and the continuity of interpolation spaces}, Trans. Amer. Math. Soc., 357 (9) (2005), 3425--3442.


\bibitem{O} N. Ozawa, {\em   An Invitation to the Similarity Problems (after Pisier)}, Surikaisekikenkyusho Kokyuroku, 1486 (2006), 27--40.

\bibitem{pelczynski} A. Pe\l czy\'nski and S. Rolewicz, {\em Best norms with respect to isometry groups in normed linear spaces}, Short communication on International Mathematical Congress in Stockholm (1964), 104.

\bibitem{P}G. Pisier, {\em Similarity problems and completely bounded maps. Second, expanded edition. Includes the solution to ``The Halmos problem''},  Lecture Notes in Mathematics, 1618. Springer-Verlag, Berlin, 2001.

\bibitem{PS}T. Pytlic and R. Szwarc, {\em An analytic family of uniformly bounded representations of free groups},
Acta Math. 157 (1986) 287--309.


\bibitem{rochberg} R. Rochberg, \emph{Higher order estimates in complex interpolation theory}, Pacific J. Math. 174 (1996), no. 1, 247--267.

\bibitem{RW} R. Rochberg, G. Weiss, \emph{Derivatives of Analytic Families of Banach
Spaces}, Annals of Math. 118 (1983) 315--347.

\bibitem{RT} G.O. Thorin. {\em An extension of a convexity theorem due to M. Riesz}. Kungl. Fysiogr. S\"allsk. i LundF\"orh., v.8 (1938), 166--170.

\bibitem{watbled}
F. Watbled, {\em Complex interpolation of a Banach space with its dual}, Math. Scand. 87 (2000) 200--210.
\end{thebibliography}
\end{document}